\documentclass[12pt]{amsart}
\usepackage[margin=1in]{geometry}
\usepackage{amssymb}
\usepackage{tikz}
\usetikzlibrary{patterns}

\usepackage{algpseudocode}

\makeatletter
\@namedef{subjclassname@2020}{\textup{2020} Mathematics  Subject  Classification}
\makeatother

\usepackage{todonotes}

\usepackage {hyperref}
\hypersetup {
    colorlinks = true,
    linktoc = all,
    linkcolor = blue,
    citecolor = red
}

\newtheorem{theorem}{Theorem}[section]
\newtheorem{lemma}[theorem]{Lemma}
\newtheorem{proposition}[theorem]{Proposition}
\newtheorem{corollary}[theorem]{Corollary}

\newtheorem{conjecture}[theorem]{Conjecture}
\newtheorem{question}[theorem]{Question}

\theoremstyle{definition}
\newtheorem{definition}[theorem]{Definition}

\newtheorem{algorithm}[theorem]{Algorithm}

\theoremstyle{remark}
\newtheorem{remark}[theorem]{Remark}

\numberwithin{equation}{section}

\newcommand{\op}{\star}
\newcommand{\XP}{\textsf{XP}}

\newcommand{\downset}[1]{\langle #1 \rangle_\downarrow}

\begin{document}

\title[Gorenstein braid cones]{Gorenstein braid cones and crepant resolutions}


\author{Joshua Hallam}
\address{Department of Mathematics, Loyola Marymount University, Los Angeles, CA 90045 USA}
\email{joshua.hallam@lmu.edu}

\author{John Machacek}
\address{Department of Mathematics, University of Oregon, Eugene, OR 97403 USA}
\email{johnmach@uoregon.edu}
\thanks{The second author acknowledges support of NSF grant DMS-2039316.}




\begin{abstract}
To any poset $P$, we associate a convex cone called a braid cone. We also associate a fan and study the toric varieties the cone and fan define.
The fan always defines a smooth toric variety $X_P$, while the toric variety $U_P$ of the cone may be singular.
We show that $X_{P} \dashrightarrow U_{P}$ is a crepant resolution of singularities if and only if $P$ is bounded.
Next, we aim to determine when $U_P$ is Gorenstein or $\mathbb{Q}$-Gorenstein.
We prove that whether or not $U_P$ is ($\mathbb{Q}$)-Gorenstein depends only on the biconnected components of the Hasse diagram of $P$.
In the case that $P$ has a minimum or maximum element, we show that the Gorenstein property of $U_P$ is completely determined by the M\"obius function of $P$. We also provide a recursive method  that determines if $U_P$ is ($\mathbb{Q}$)-Gorenstein  in this case.  
We conjecture that $U_P$ is Gorenstein if and only if it is $\mathbb{Q}$-Gorenstein.
We verify this conjecture for posets of length $1$ and also for posets with a minimum or maximum element.
\end{abstract}

\maketitle
\tableofcontents

\section{Introduction}
The correspondence between convex cones and toric varieties is well established.
We consider posets as an additional combinatorial layer.
Every poset naturally gives rise to a  cone   as well as a fan that refines the cone.   The cones that arise here are the unions of certain chambers of the braid arrangement and thus called \emph{braid cones}. We use the combinatorics of the underlying posets and their Hasse diagrams to study affine toric varieties associated to braid cones.  We also study a resolution of singularities for such a toric variety coming from another toric variety associated to a fan.

Our main focus is to be able to efficiently decide when the variety associated to a braid cone is Gorenstein and when the aforementioned resolution is crepant.
Since we are working with affine toric varieties, this algebro-geometric decision problem ends up being equivalent to the discrete geometry problem of deciding if there exists a hyperplane containing certain lattice points. 
We find and use properties of posets which are equivalent to the desired resolutions of singularities being crepant and the Gorenstein property.

Crepant resolutions and the Gorenstein property are important in studying singularities in toric geometry (see e.g.~\cite{Cox}). Crepant resolutions have no discrepancy in the canonical class and play a role in the minimal model program.  The class of Gorenstein varieties consists of varieties which are at worse ``not too singular''. It properly sits between the class of smooth varieties and   Cohen-Macaulay varieties. That is,
$$
\mbox{smooth varieties} \subsetneq \mbox{Gorenstein varieties} \subsetneq \mbox{Cohen-Macaulay varieties}.
$$
All  toric varieties coming from fans  are normal and hence by~\cite{Hochster} are Cohen-Macaulay. Moreover, in~\cite{PRW}, a classification of smoothness for toric varieties arising from braid cones  is given.

Recent works in a similar spirit to what we study here include~\cite{graphs}, \cite{multi} and \cite{matroids}. In~\cite{graphs} the authors give graph-theoretic characterizations of Gorensteinness for the toric variety  associated to  the base polytope of a graphic matroid arising from (simple) graphs. The work in~\cite{multi} does the same in the case of graphic matroids coming from multigraphs. In~\cite{matroids}, a matroid-theoretic characterization for Gorensteinness of the varieties  associated to the base polytopes of general matroids is given.

Cones corresponding to posets have been considered in a generalized setting of root systems for Coxeter groups under the names \emph{parsets}~\cite{Reiner} and \emph{Coxeter cones}~\cite{Stem}.
The case of posets in this setting is the Type $A$ case.
Recent work on cones coming from posets include~\cite{Whitney}  and~\cite{Mac}. The former focuses on their Whitney numbers  and the latter studies  star subdivisions which correspond to toric blow-ups.

In the remainder of this section, we review posets and braid cones.
Afterwards, we discuss the necessary background on toric varieties.
In Section~\ref{sec:label} we introduce labelings of posets which will determine when a poset defines a ($\mathbb{Q}$)-Gorenstein toric variety as well as when a particular fan associated to a poset gives a crepant resolution of singularities. 
Our main result on crepant resolutions is in Section~\ref{sec:crepant}. In Theorem~\ref{thm:crepant}, we show there is a certain crepant resolution of singularities respecting Weyl chambers precisely for toric varieties coming from bounded posets (i.e.~posets with both a minimum and maximum).
In Section~\ref{sec:Gor} we deal with the Gorenstein property and find, in Theorem~\ref{thm:bicon}, that Gorensteinness depends only on the biconnected components of the Hasse diagram.
We consider posets with a minimum or a maximum element in Section~\ref{sec:posetsWith0or1}. There, we are able to characterize when such posets give rise to a Gorenstein variety (Theorem~\ref{thm:characterizationOfRankedPosetGor}) in a way which gives a recursive algorithm (Algorithm~\ref{alg:RankedPosAlg}) to do this. 
In terms of parameterized complexity, this algorithm is an \XP{}-algorithm (i.e.~``slicewise polynomial'') implying it runs efficiently for small values of the parameter.  The parameter used is a new parameter on posets which depends only on the minimal elements as well as those elements which cover only minimal elements.
The M\"obius function plays a significant role in developing the ideas in this section.
We then look at posets of length 1 in Section~\ref{sec:len1} and conclude with a discussion on some open problems in Section~\ref{sec:conclusion}.

\subsection{Posets, connectedness, and braid cones}
A much fuller treatment of braid cones can be found in~\cite{PRW}.
Here we only discuss what is needed for our purposes.
We write $[n] := \{1,2,\dots,n\}$ and have the lattice $\mathbb{Z}^n \subseteq \mathbb{R}^n$ with standard basis denoted $\{e_i \mid i \in [n]\}$ and coordinates $(x_1, \dots, x_n)$.
For each $A \subseteq [n]$, we define $e_A := \sum_{i \in A} e_i$.
Our focus will be on the lattice $N = \mathbb{Z}^n / \mathbb{Z} e_{[n]}$ with corresponding vector space $N_{\mathbb{R}} = \mathbb{R} \otimes_{\mathbb{Z}} N$.
The dual lattice and vector space will be denoted by $M$ and $M_{\mathbb{R}}$ respectively.

We will assume that the reader is familiar with basic properties of posets.  For  more information on posets or any undefined terms, the reader may consult~\cite[Chapter 5]{SaganCAOC} and~\cite[Chapter 3] {StanleyEC1}. We will often need to refer to graph-theoretic aspects of the Hasse diagram of a poset. Unless otherwise noted, when we apply graph-theoretic adjectives to a poset or its Hasse diagram, we are viewing the Hasse diagram as an undirected graph.  For example, when we say a poset $P$ is \emph{connected} or \emph{biconnected}, we mean that the underlying undirected Hasse diagram of $P$ is connected or biconnected.

The \emph{length} of a chain in a poset is one less the number of elements in the chain. For example, the  chain $x_0 < x_1 < \cdots < x_{\ell}$ has length $\ell$.
The \emph{length} of a poset $P$ is the maximum among all lengths of chains in $P$.
We will use $\ell(P)$ to denote the length of a poset.
Note that in the case that $P$ is ranked, the length is the rank of the poset.

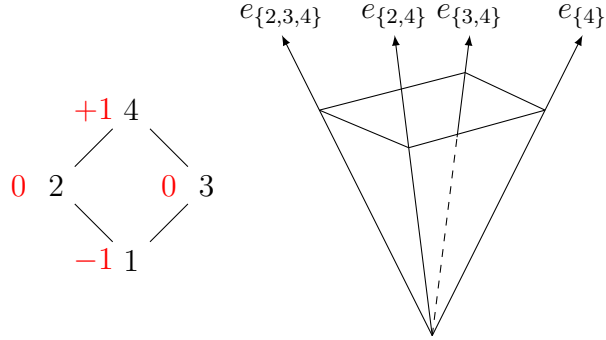
\begin{figure}
\centering
    \begin{tikzpicture}

\node (4) at (-4,3) {$4$};
\node (2) at (-5,2) {$2$};
\node (3) at (-3,2) {$3$};
\node (1) at (-4,1) {$1$};

\node at (-4-.5,3) {\textcolor{red}{$+1$}};
\node at  (-5-.5,2){\textcolor{red}{$0$}};
\node at  (-3-.5,2){\textcolor{red}{$0$}};
\node at  (-4-.5,1) {\textcolor{red}{$-1$}};


\draw (4)--(2)--(1)--(3)--(4);

    \draw[-{latex}] (0,0) -- (-2,4);
    \node at (-2,4.25) {$e_{\{2,3,4\}}$};
    \draw[dashed] (0,0) -- (0.33, 8*0.33);
    \draw[-{latex}] (0.33, 8*0.33) -- (0.5,4);
    \node at (0.5,4.25) {$e_{\{3,4\}}$};
    \draw[-{latex}] (0,0) -- (-0.5,4);
    \node at (-0.5,4.25) {$e_{\{2,4\}}$};
    \draw[-{latex}] (0,0) -- (2,4);
    \node at (2,4.25) {$e_{\{4\}}$};
    \draw (-1.5,3) -- (0.4375,3.5);
    \draw (-1.5,3) -- (-0.3125,2.5);
    \draw (-0.3125,2.5) -- (1.5,3);
    \draw (0.4375,3.5) -- (1.5,3);

\end{tikzpicture}
\caption{The Hasse diagram of a poset $P$ along with the cone $\sigma_P$. The values in red give  a chamber crepant labeling of $P$.}
\label{fig:ex_cone}
\end{figure}

Let $P$ be a connected poset with with underlying set $[n]$.  The \emph{braid cone}  associated to $P$, denoted by $\sigma_P$, is given by the  intersection of the half-spaces $x_i \leq x_j$ for each pair $(i,j)$ with $i <_P j$.
Note that, by~\cite[Proposition 3.5 (8)]{PRW},  it suffices to only use the half-spaces  $x_i \leq x_j$  when $i\lessdot_P j$. In Figure~\ref{fig:ex_cone} we have an example of a poset $P$ and the cone $\sigma_P$.
In this example, the cone is the intersection of the four half-spaces $x_1 \leq x_2$, $x_1 \leq x_3$, $x_2 \leq x_4$, and $x_3 \leq x_4$.
One can check  that $\sigma_P$ is also given by the positive span of the vectors $e_{\{4\}}$, $e_{\{2,4\}}$, $e_{\{3,4\}}$, and $e_{\{2,3,4\}}$.
We will give a general way to find these generating vectors in Lemma~\ref{lem:latticept}.

In~\cite[Proposition 3.5]{PRW} it is shown that for any poset $P$, the cone $\sigma_P$ is full dimensional and rational. The connectedness of $P$ guarantees that $\sigma_P$ is strongly convex (which the authors call \emph{pointed} in~\cite{PRW}).
We will assume throughout the article that all posets are connected and also assume that all cones encountered are  strongly convex, rational polyhedral cones.  We will also assume that $|P|\geq 2$ to avoid issues with triviality.
It is often convenient to assume our posets have $[n]$ as their underlying set.
However, if $P$ is any poset with an underlying set of cardinality $n$, we may pick some bijection to $[n]$ and use this bijection to define $\sigma_P$.
None of the properties of $\sigma_P$ that we are interested in are influenced in any way by this choice of bijection.
This will allow us  to consider multiple posets with disjoint underlying sets.

The \emph{braid arrangement} is the hyperplane arrangement in $N_{\mathbb{R}}$ consisting of the hyperplanes $x_i = x_j$ for $1 \leq i < j \leq n$. This  is the Coxeter arrangement of type $A_{n-1}$.
The regions of the braid arrangement are the Weyl chambers of type $A_{n-1}$ and are $\sigma_L$ for each linear order $L$ on the elements of $[n]$.
We use $|\sigma|$ to denote the \emph{support} of a cone $\sigma$ which consists of all points of $N_{\mathbb{R}}$ that are contained in $\sigma$.
In general $|\sigma_P|$ is  the union of $|\sigma_L|$ over all linear extensions $L$ of $P$.
For any poset $P$, we write $\Sigma_P$ for the fan consisting of the cones $\sigma_L$ and its faces for each linear extension $L$ of $P$.

Returning to our example of $P$ in Figure~\ref{fig:ex_cone}, we can obtain $\Sigma_P$ by adding a wall on the hyperplane $x_2 = x_3$ which is  spanned by $e_{\{2,3,4\}}$ and $e_{\{4\}}$. The fan will then have two maximal cones, each corresponding to one of the two linear extensions of $P$.

\subsection{The corresponding toric varieties}
For an in-depth introduction to toric varieties and the Gorenstein property, the reader can see the text~\cite{CLS}, particularly Chapter 8 of the book.  Given a cone $\sigma$, which is minimally generated by the positive span of finitely many vectors $\{v_i\}_i$, the set of \emph{ray generators} of $\sigma$ is $\{w_i\}_i$ where $w_i$ is the first nonzero lattice point on the ray $\mathbb{R}_{\geq 0} v_i$. Note that the ray generators of a cone are unique.  If $\sigma$ is strongly convex  with ray generators from some lattice, we get an affine toric variety we denote by $U_{\sigma}$. Similarly, given any fan of such cones $\Sigma$, it defines an abstract toric variety denoted $X_{\Sigma}$. We will write $U_P$ for $U_{\sigma_P}$ and $X_P$ for $X_{\Sigma_P}$.

A variety is called \emph{$\mathbb{Q}$-Gorenstein} if some multiple of its canonical divisor is Cartier.
A variety is \emph{Gorenstein} if it is Cohen-Macaulay and its canonical divisor is Cartier. When the variety is toric and defined by a cone, there is an equivalent definition of ($\mathbb{Q}$)-Gorenstein as we explain next.  Let $v_1, v_2, \dots, v_k \in N$ be the ray generators of $\sigma$, then $U_{\sigma}$ is $\mathbb{Q}$-Gorenstein if and only if the exists $u \in M$ such that $\langle u, v_i \rangle = r$ for all $1 \leq i \leq k$ and some positive integer $r$.
The \emph{index} of a $\mathbb{Q}$-Gorenstein toric variety is the minimal possible $r$ which can be taken.
If the index is $r=1$, then the toric variety is Gorenstein since toric varieties arising from fans are always Cohen-Macaulay~\cite{Hochster}.

A cone is \emph{smooth} if its ray generators can be extended to a basis for the lattice, and a fan is smooth if it consists of smooth cones.
If $\Sigma$ is a smooth fan refining a cone $\sigma$ then $X_{\Sigma} \dashrightarrow U_{\sigma}$ gives a resolution of singularities. 
Take a $\mathbb{Q}$-Gorenstein toric variety $U_{\sigma}$ of index $r$ and $u \in M$ such that $\langle u, v_i \rangle = r$ for all ray generators $v_1, v_2, \dots, v_k$ of $\sigma$.
If $\Sigma$ is a smooth fan refining $\sigma$ such that any ray generator $v$ of $\Sigma$ also has $\langle u, v \rangle = r$, then this resolution is called \emph{crepant}.
A crepant resolution does not change the canonical class.

As we saw earlier, the cone   in Figure~\ref{fig:ex_cone} is generated by  $e_{\{4\}}$, $e_{\{2,4\}}$, $e_{\{3,4\}}$, and $e_{\{2,3,4\}}$. In fact, these are the ray generators of the cone.  If we take $u = e^*_{\{4\}}-e^*_{\{1\}} $, then 
$$
\langle u, e_{\{4\}} \rangle = \langle u, e_{\{2,4\}} \rangle =  \langle u, e_{\{3,4\}} \rangle = \langle u,e_{\{2,3,4\}} \rangle =1.
$$
Moreover,
$$
\langle u, e_{\{1,2,3,4\}} \rangle =0
$$
which must hold since $e_{\{1,2,3,4\}} =0$ in $N_{\mathbb{R}}$. We conclude that $U_P$ is Gorenstein. Because all the ray generators of the fan $\Sigma_P$ are also ray generators of $U_P$,  $X_{P} \dashrightarrow U_{P}$  is crepant.

At this point the reader may be wondering how one can find a $u\in M$ to determine if $U_P$ is Gorenstein or if $X_{P} \dashrightarrow U_{P}$  is crepant.  In the next section we introduce poset labelings to do this exactly.  This allows us to translate our  algebraic problem   to a purely combinatorial one

\section{Ray generators, lattice points, and labelings}
\label{sec:label}
In this section we describe the ray generators of a braid cone as well as the lattice points of the form $e_A$ for $A \subseteq [n]$ which are contained in the cone.
We use this description to define labelings of the poset $P$ which capture when $U_P$ is ($\mathbb{Q}$-)Gorenstein and when $X_P \dashrightarrow U_P$ is crepant.

For a binary relation $R$, let $R^{op}$ denote the opposite binary relation where $(i,j) \in R^{op}$ if and only if $(j,i) \in R$.
We define a \emph{contraction} of a poset $P$ to be the transitive closure of $P \cup R^{op}$ for some $R \subseteq P$.
For example, taking the poset $P = \{2 <_P 1, 2 <_P 3\}$ and $R = \{2 <_R 1\}$ we find the transitive closure of $P \cup R^{op}$ to be  $\{1 < 2, 2 < 1, 1 < 3, 2 < 3\}$ which is now not a poset, but is a preposet.
A key fact is that if a poset $P$ corresponds to a cone $\sigma$, then a preposet $P'$ corresponds to a face $\tau \subseteq \sigma$ if and only if $P'$ is a contraction of $P$~\cite[Proposition 3.5 (2)]{PRW}.

\begin{remark}
A related notion to the opposite binary relation, is the \emph{order dual} of a poset $P$. This is given by reversing all the inequalities defining $P$.
To avoid confusion, we will use the notation $P^*$ when we are dealing taking the order dual of an entire partial order $P$ whereas $R^{op}$ can be used for any binary relation. 
\end{remark}

Given an upset $A$ of $P$, we define
\[R_A := \{(i,j) \in P \mid i,j \in A \text{ or } i,j \not\in A\}\]
and let the contraction $P_A$ be the  transitive closure of $P \cup R_A^{op}$.  The fact that $P_A$ is a preposet means that $P_A$ is reflexive and transitive.
In particular, $P_A$ need not be antisymmetric.
So, $P_A$ defines an equivalence relation where $i$ is equivalent to $j$ if $(i,j) \in P_A$ and $(j,i) \in P_A$.
We also define the \emph{dimension} of $A$, denoted $\dim(A)$, to be one less than the number of equivalence classes determined by the preposet $P_A$.
Note the $\dim(A)$ depends on the poset $P$, but $A$ is always taken to be an upset of a poset $P$ so the data on the poset is present.
The dimension, $\dim(A)$, can be computed by
\[\dim(A) = cc(A) + cc(\overline{A}) - 1\]
where $cc(A)$ and $cc(\overline{A})$ are the number of connected components of the Hasse diagram of $P$ restricted to $A$ and $\overline{A}$ respectively.
Here $\overline{A}$ denotes the complement of $A$ in $[n]$. For the poset depicted in Figure~\ref{fig:basicGorLabExample}, we have that $\dim(\{1,3\})=1$ since $\{1,3\}$ and $\overline{\{1,3\}} =\{2\}$ are both connected, whereas $\dim(\{3\})=2$ since  $\{3\}$  is connected and removing it disconnects the graph into two parts.

Since we are interested in the Gorenstein property and crepant resolutions, we need to know what the ray generators of a cone $\sigma_P$ are as well as which ray generators are added when refining $\sigma_P$ to $\Sigma_P$.
The next lemma will describe primitive lattice points on these rays.
Let us briefly recall a few definitions from convex geometry.
A \emph{face} of a cone is a subset of the cone obtained by intersection with a supporting hyperplane.
The \emph{dimension} of a face is the dimension of its linear span.
Lastly, a point is said to be in the \emph{relative interior} of a face $\tau$ if it is not contained in any proper face $\tau' \subsetneq \tau$.

\begin{lemma}\label{lem:rayGen}
If $P$ is a poset, then the lattice point $e_A$ is contained in $\sigma_P$ if and only if $A$ is an upset.
Moreover, for an upset $A$ the lattice point $e_A$ is contained in the relative interior of a face $\tau \subseteq \sigma_P$ with $\dim(\tau) = \dim(A)$.  In particular, $e_A$ is a ray generator of $\sigma_P$ if and only if $A$ is an upset of dimension 1.
\label{lem:latticept}
\end{lemma}
\begin{proof}
It follows immediately from the definitions that $e_A \in \sigma_P$ if and only if $A$ is an upset.
Now assume that $A$ is an upset so that $e_A \in \sigma_P$.
Then there is a face $\tau \subseteq \sigma_P$ with $\dim(\tau) = \dim(A)$ indexed by the contraction $P_A$.
Furthermore, we have $e_A \in \tau$ by construction. 
It only remains to show that $e_A$ is in the relative interior of $\tau$.
That is, we must show $e_A$ is not in any face properly contained in $\tau$.
Such faces properly contained in $\tau$ will correspond to contractions of $P_A$.
Take any $(i,j) \in P_A$.
If $(j,i) \in P_A$ this relation will not matter in any further contractions.
If $(j,i) \not\in P_A$, then since  $A$ is an upset we have that $j \in A$ and $i \not \in A$.
For any contraction which is the transitive closure of $P_A \cup R^{op}$ with $(i,j) \in R$, we must have $x_j \leq x_i$ in the corresponding cone.
However, since $j \in A$ and $i \not\in A$ we see that $e_A$ does not satisfy this inequality.
Therefore $e_A$ is in the relative interior of $\tau$ as desired. 

The final assertion now follows since the  ray generators that show up in the braid arrangement fan are precisely all $0$-$1$ vectors except the all $0$ vector $e_{\varnothing}$ and all the $1$ vector $e_{[n]}$ (which is equal to zero is our setting).
\end{proof}

\begin{figure}\label{fig:basicGorLabExample}
    \begin{tikzpicture}


    

\node (3) at (-4,3) {$3$};
\node (1) at (-5,2) {$1$};
\node (2) at (-3,2) {$2$};

\node at (-4,1) {$Q$};

\node at (-4-.5,3) {\textcolor{red}{$+2$}};
\node at  (-5-.5,2){\textcolor{red}{$-1$}};
\node at  (-3-.5,2){\textcolor{red}{$-1$}};

\draw (1)--(3);
\draw (2)--(3);
    \end{tikzpicture}
    
    \caption{A poset together and its Gorenstein labeling given in red.}
\end{figure}
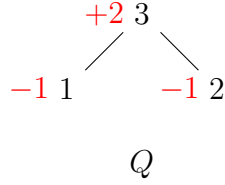

Let us now determine if $U_Q$ is Gorenstein for $Q$ in Figure~\ref{fig:basicGorLabExample}.  Recall that $U_Q$ is Gorenstein if and only if there is a $u\in M$ such that $\langle u, e_A\rangle =1$ for all ray generators $e_{A}$.    Let $u = y_1e^*_{\{1\}} +y_2e^*_{\{2\}}+y_3e^*_{\{3\}}$.  Using Lemma~\ref{lem:rayGen}, one can verify that the ray generators of $\sigma_Q$ are $e_{\{1,3\}}$ and $e_{\{2,3\}}$.  This together with the fact that $e_{\{1,2,3\}}=0$ in $N_{\mathbb{R}}$  implies that if $U_Q$ is Gorenstein, then the following system has a solution. 
\begin{align*}
    y_1+ y_3&= 1\\
    y_2+ y_3 &=1\\
    y_1+y_2+y_3&=0
\end{align*}
The system does have a solution namely $y_1=-1, y_2=-1, y_3=2$, and so $U_Q$ is Gorenstein.  Since $1,2$, and $3$ are elements of $Q$, we can ``label" the element $i$ of $Q$ with the value of $y_i$ to get a labeling that completely determines if $U_Q$ is Gorenstein.    Motivated by the ideas presented in this example, we introduce the following poset labelings.

\begin{definition}\label{def:gorLab}
Given a poset $P$ on $[n]$, an \emph{$r$-Gorenstein labeling} is a function $\phi:P \to \mathbb{Z}$ such that
\[\sum_{i \in P} \phi(i) = 0\]
and
\[\sum_{i \in A} \phi(i) = r\]
for all upsets $A$ with $\dim(A) = 1$.
We call a $1$-Gorenstein labeling a \emph{Gorenstein labeling}.
\end{definition}

\begin{remark}
A poset admitting an $r$-Gorenstein labeling is unrelated to a poset being Gorenstein*.
As the next proposition shows, an $r$-Gorenstein labeling of a poset $P$ coincides with the toric variety $U_P$ being ($\mathbb{Q}$)-Gorenstein.
Whereas the Gorenstein* property is related to Gorensteinness of the order complex of $P$~\cite{StanGor}.
\end{remark}

One can verify that the labelings in Figure~\ref{fig:ex_cone} and Figure~\ref{fig:basicGorLabExample} are   Gorenstein labelings.   As we saw earlier, the  varieties associated to these posets are both Gorenstein.  This is no coincidence as we see next.

\begin{proposition}\label{prop:GorIffGorLab}
Let $P$ be a poset.
$U_{P}$ is $\mathbb{Q}$-Gorenstein with index $r$ if and only if  $P$ has an $r$-Gorenstein labeling but has no $s$-Gorenstein labeling for any $0 < s < r$.
\label{prop:Gorenstein}
\end{proposition}
\begin{proof}
Given a poset $P$, let 
\[\{A_1, A_2, \dots, A_k\} = \{A \mid A \text{ is an upset and } \dim(A) = 1\}.\]
So by Lemma~\ref{lem:latticept}, the ray generators of $\sigma_P$ are $e_{A_1}, e_{A_2}, \dots, e_{A_k}$.
We observe that $u \in M$ with $\langle u, e_{A_j} \rangle = r$ for $1 \leq j \leq k$ is equivalent to an $r$-Gorenstein labeling $\phi$ where $\phi(i) = \langle u, e_i \rangle$ for $1 \leq i \leq n$.
\end{proof}

\begin{definition}
Let $P$ be a poset for which $U_{P}$ is $\mathbb{Q}$-Gorenstein with index $r$.
A \emph{chamber crepant labeling} of $P$ is a function $\phi:P \to \mathbb{Z}$ such that
\[\sum_{i \in P} \phi(i) = 0\]
and
\[\sum_{i \in A} \phi(i) = r\]
for all upsets $A$ with $\dim(A) > 0$.
\end{definition}

One can check that the labeling of the poset in Figure~\ref{fig:ex_cone} is a chamber crepant labeling.  Indeed, every upset except the full poset only contains a single element with a nonzero label and the total sum of values is 0.  On the other hand, the labeling of the poset in Figure~\ref{fig:basicGorLabExample} is not a chamber crepant labeling. For example, the sum of the values in the upset $\{1,3\}$ is 1 and the sum of values for $\{3\}$ is 2.  As we will see in Section~\ref{sec:crepant}, there are no chamber crepant labelings of the poset in Figure~\ref{fig:basicGorLabExample} because it is not bounded.

\begin{proposition}\label{prop:CrepIffGorLab}
Let $P$ be a poset for which $U_{P}$ is $\mathbb{Q}$-Gorenstein.
The poset $P$ has a chamber crepant labeling if and only if $X_{P} \dashrightarrow U_{P}$ is crepant.
\label{prop:crepant}
\end{proposition}
\begin{proof}
Having a chamber crepant labeling means there is a $u \in M$ such that $\langle u, e_A \rangle = r$ for some positive integer $r$ and all upsets $A$ with $\dim(A) > 0$.
By Lemma~\ref{lem:latticept}, we see that the rays added when refining $\sigma_P$ to $\Sigma_P$ will be exactly $e_A$ such that $A$ is an upset with $\dim(A) > 0$.
So, having a chamber crepant labeling is equivalent to $U_p$ being $\mathbb{Q}$-Gorenstein and $X_{P} \dashrightarrow U_{P}$ being crepant.
\end{proof}

Before we move on, let us note that if $P$ has a $r$-Gorenstein labeling, then it only has one such labeling.  Indeed, by~\cite[Proposition 3.5 (4)]{PRW},  $\sigma_P$ is a full-dimensional cone. Thus the set  $\{e_{A} \mid \dim(A) =0,1\}$ has full rank and so the linear equations given in Definition~\ref{def:gorLab} have at most one solution.  Since the chamber crepant labelings are special cases of $r$-Gorenstein labeling (where $r$ is the index of the variety), we see that if there is a chamber crepant labeling, it is also unique.

As one can see from Lemma~\ref{lem:latticept}, once we know the upsets of a poset we can determine the ray generators of the cone. Then determining whether the corresponding toric variety is  Gorenstein or if it has a crepant resolution amounts to solving a system of linear equations. 
Proposition~\ref{prop:GorIffGorLab} and Proposition~\ref{prop:CrepIffGorLab} show we can alternatively work with labelings of the poset to solve this system.
Of course solving a linear system is easy, the  real difficulty in these problems stems from the fact that determining the upsets of a general poset is known to be difficult.
The problem of counting upsets has been shown to be \textsf{\#P}-complete~\cite{ProvanBall1983}. 

Since the ray generators correspond to dimension 1 upsets, we do not need to actually know all the upsets to determine if the variety is Gorenstein.
However, even finding the ray generators does not have an easy known solution in general. 
The cover relations of $P$, which are the edges of the Hasse diagram, give the half-space representation of polyhedron $\sigma_P$. The problem of finding the ray generators of $\sigma_P$ is equivalent to finding the vertex representation of a certain polytope.  This polytope is the projectivization that one obtains by adding the ``far face'' at infinity to $\sigma_P$. Vertex enumeration of general polyhedra given by half-space representation is known to be a hard problem~\cite{hard}. 
 
Despite the difficulties outlined above, we are able to show that under certain circumstances, there are efficient methods to solve these problems. This is what we explore in the remainder of the paper.

\section{Chamber crepant labelings and bounded posets}
\label{sec:crepant}

In~\cite{Mac} it was shown that, for any poset $P$, $\Sigma_P$ can always be obtained from $\sigma_P$ by a sequence of star subdivisions.
This means that $X_P$ and $U_P$ are related by a sequence of blow-ups.
In this section we further look at the birational map $X_P \dashrightarrow U_P$ and determine when  it is crepant.  First, we  explore how the Gorenstein and crepant property behave when taking order duals.

 If $A$ is a  upset of $P$, then $\overline{A}$ is a downset of $P$.  Since the total sum of labels in a $r$-Gorenstein labeling must be 0, this means that the sum of the labels of $A$ is completely determined by the sum of labels of $\overline{A}$ when $A$ and $\overline{A}$ are both connected.  As result, we get the following.

\begin{proposition}\label{prop:downsetLabel}
Let $P$ be a poset. The map $\phi:P \to \mathbb{Z}$ is a $r$-Gorenstein (resp.~chamber crepant) labeling if and only if 
\[\sum_{i \in P} \phi(i) = 0\]
and
\[\sum_{i \in A} \phi(i) = -r\]
for all downsets $A$  such $\dim(A)=1$  (resp.~$\dim(A)>0$). 
\end{proposition}

Upsets and downsets reverse roles when taking the order dual of a poset. Thus the previous proposition implies that if $\phi$ is an $r$-Gorenstein (resp.~chamber crepant) labeling of $P$, then $-\phi$ is a $r$-Gorenstein (resp.~chamber crepant) labeling of the order dual $P^*$.  This gives the following.

\begin{proposition}\label{prop:dual}
Let $P$ be a poset and let $P^*$ be its order dual.  $U_P$ is ($\mathbb{Q}$)-Gorenstein if and only if $U_{P^*}$ is ($\mathbb{Q}$)-Gorenstein.  Moreover, $X_P\dashrightarrow U_P$ is crepant if and only if $X_{P^*} \dashrightarrow U_{P^*}$ is crepant.
\end{proposition}

Recall a poset $P$ is \emph{bounded} if $P$ has a minimal element (denoted by $\hat{0}$) and a maximal element (denoted by $\hat{1}$).

\begin{theorem}\label{thm:crepant}
Let $P$ be a poset. $X_P\dashrightarrow U_P$ is crepant if and only if $P$ is bounded.
\end{theorem}
\begin{proof}
($\Rightarrow$) By Definition~\ref{def:gorLab}, if $\phi$ is an chamber crepant labeling of $P$ and $m$ is a maximal element of $P$, then $\phi(m)=r$. Now let $A$ be the set of maximal elements of $P$.  Then $A$ is an upset and since $P$ is connected and non-trivial, $A\neq P$. So if $X_P\dashrightarrow U_P$ is crepant, the sum of the labels of $A$ must be $r$ and so $|A|=1$.  It follows that $P$ has a $\hat{1}$.  By Proposition~\ref{prop:dual}, the crepant property is preserved by taking order duals. Using the fact that  a poset has a $\hat{1}$ if and only if its order dual has a $\hat{0}$, we get this direction.

($\Leftarrow$) Now suppose that $P$ is bounded.  Define $\phi: P \to \mathbb{Z}$ by $\phi(\hat{0}) =-1$, $\phi(\hat{1})=1$, and $\phi(x)=0$ for all $x\in P\setminus \{\hat{0},\hat{1}\}$.  It is straightforward to show that $\phi$ is a chamber crepant labeling and so we have this direction.
\end{proof}

\section{Gorensteinness and biconnected components}
\label{sec:Gor}

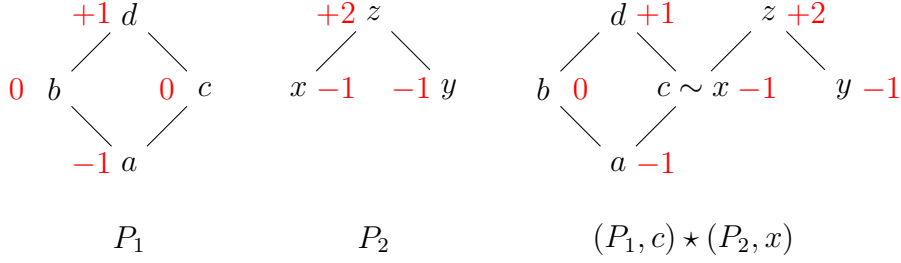
\begin{figure}
\centering
    \begin{tikzpicture}
\node (4) at (-4,3) {$d$};
\node at (-4-.5,3) {\textcolor{red}{$+1$}};
\node (2) at (-5,2) {$b$};
\node at  (-5-.5,2){\textcolor{red}{$0$}};
\node (3) at (-3,2) {$c$};
\node at  (-3-.5,2){\textcolor{red}{$0$}};
\node (1) at (-4,1) {$a$};
\node at  (-4-.5,1) {\textcolor{red}{$-1$}};

\draw (4)--(2)--(1)--(3)--(4);

\node at  (-4,0) {$P_1$};
 
 \begin{scope}[shift={(3.25,0)}]
  \node (4) at (-4,3) {$z$};
  \node at (-4-.5,3) {\textcolor{red}{$+2$}};
\node (2) at (-5,2) {$x$};
\node at  (-5+.5,2){\textcolor{red}{$-1$}};
\node (3) at (-3,2) {$y$};
\node at  (-3-.5,2){\textcolor{red}{$-1$}};

\draw (4)--(2);
\draw (3)--(4);
 \node at  (-4,0) {$P_2$};

 \end{scope}

 \begin{scope}[shift={(6.5,0)}]
\node (4) at (-4,3) {$d$};
\node at (-4+.5,3) {\textcolor{red}{$+1$}};

\node (2) at (-5,2) {$b$};
\node at  (-5+.5,2){\textcolor{red}{$0$}};

\node (3) at (-3,2) {$c\sim x$};
\node at  (-3+0.85,2){\textcolor{red}{$-1$}};

\node (1) at (-4,1) {$a$};
\node at  (-4+.5,1) {\textcolor{red}{$-1$}};

\draw (4)--(2)--(1)--(3)--(4);

 \node (6) at (-1,2){$y$};
 \node at  (-1+.5,2){\textcolor{red}{$-1$}};

 \node (7) at (-2,3) {$z$};
   \node at (-2+.5,3) {\textcolor{red}{$+2$}};

 \draw (3)--(7)--(6);
 
 \node at  (-3,0) {$(P_1, c) \op (P_2, x) $};
 \end{scope}
 
\end{tikzpicture}
\caption{Posets  $P_1$, $P_2$ and  $(P_1, c) \op (P_2, x) $ along with Gorenstein labelings in red.}
\label{fig:ex_glue}
\end{figure}

Given any two  disjoint posets $P_1$ and $P_2$ with Hasse diagrams $H_1$ and $H_2$ respectively along with $x_1 \in P_1$ and $x_2 \in P_2$ we define
\[(P_1, x_1) \op (P_2, x_2) := \left(P_1 \uplus P_2\right)/(x_1 \sim x_2)\]
which can be gotten by taking disjoint copies of $H_1$ and $H_2$ and identifying $x_1$ with $x_2$. Informally, we can think of this as ``gluing" together the posets along $x_1$ and $x_2$. See Figure~\ref{fig:ex_glue} for an example. 

In the example depicted in Figure~\ref{fig:ex_glue}, the labelings of $P_1$ and $P_2$ are Gorenstein labelings. The labeling of $(P_1, c) \op (P_2, x) $ is obtained by summing the values of the elements that were identified and keeping all the other labels the same.  One can verify that this new labeling is a Gorenstein labeling. As we show next, this is no coincidence.

\begin{lemma}\label{lem:gluing}
Let $P_1$ and $P_2$ be posets with $x_1 \in P_1$ and $x_2 \in P_2$.
If both $P_1$ and $P_2$  admit $r$-Gorenstein labelings, then $(P_1, x_1) \op (P_2, x_2)$ admits an $r$-Gorenstein labeling.
\end{lemma}

\begin{proof}
Let $\phi_1$ and $\phi_2$ be $r$-Gorenstein labelings for $P_1$ and $P_2$ respectively.
Also, let $P = (P_1, x_1) \op (P_2, x_2)$. Let $\phi:P \to \mathbb{Z}$ given by
\[\phi(z) = \begin{cases}\phi_1(z) & \text{if } z \in P_1 \setminus \{x_1\}, \\ \phi_2(z) & \text{if } z \in P_2 \setminus \{x_2\}, \\ \phi_1(z) + \phi_2(z) & \mathrm{if }\, z = x_1 = x_2.\end{cases}\]
We claim that $\phi$  is an $r$-Gorenstein labeling for $P$.
Let $A \subseteq P$ be an upset with $\dim(A) = 1$.
First suppose that $A$ does not contain $x_1=x_2$. Then either $A \subseteq P_1$ and
\[\sum_{z \in A} \phi(z) = \sum_{z \in A} \phi_1(z) = r\]
or else $A \subseteq P_2$ and the result follows by same calculation with $\phi_2$ in place of $\phi_1$.
Otherwise $x_1 = x_2$ is in  $A$.
Then it must be that $P_1 \subseteq A$ or $P_2 \subseteq A$.
If not $\bar{A} \cap P_1$ and $\bar{A} \cap P_2$ are nonempty and disjoint. However this would imply that $\dim(A)>1$.
When $P_1 \subseteq A$ we have that
\begin{align*}
   \sum_{z \in A} \phi(z) &= \sum_{z \in A \cap (P_1 \setminus \{x_1\})} \phi(z) + \phi(x_1 = x_2) +  \sum_{z \in A \cap (P_2 \setminus \{x_2\})} \phi(z)\\
   &= \sum_{z \in A \cap (P_1 \setminus \{x_1\})} \phi_1(z) + \phi_1(x_1)  +\phi_2(x_2) +  \sum_{z \in A \cap (P_2 \setminus \{x_1\})} \phi_2(z)\\
   &= \sum_{z \in P_1} \phi_1(z) + \sum_{z \in A \cap P_2} \phi_2(z)\\
   &= 0+r\\
   &=r
\end{align*}
where the second to last line holds since $A\cap P_2$ is an upset of dimension 1 in $P_2$.  A similar calculation also holds in the case that $P_2\subseteq  A$.   Thus, the lemma is proven.
\end{proof}

Recall, a graph is biconnected if it is connected and remains connected after the removal of any single vertex. A \emph{biconnected component} of a graph is a maximal biconnected subgraph.
For any graph its biconnected components and cut vertices form the vertex set of the \emph{block-cut tree}. In this tree, we have an edge between a biconnected component and a cut vertex if and only if the cut vertex is contained in the biconnected component. 
The next lemma will be proven by induction using biconnected components and the block-cut tree of a Hasse diagram.

\begin{lemma}
If $P$ admits an $r$-Gorenstein labeling, then each biconnected component of $P$ admits an $r$-Gorenstein labeling.
\label{lem:bicomponent}
\end{lemma}
\begin{proof}
If $P$ is biconnected, we are done.  So we may assume it is not biconnected.
Let $x_0$ be a cut vertex such that $X \cup Y = P$, $X \cap Y = \{x_0\}$, and $P|_X$ is a leaf of the block-cut tree.
Let $P_X = P|_X$ and $P_Y = P|_Y$ both of which are connected posets.
We show both $P_X$ and $P_Y$ have $r$-Gorenstein labelings.

By assumption $P$  has an $r$-Gorenstein labeling, $\phi:P \to \mathbb{Z}$. Define $\phi_X: P_X \to \mathbb{Z}$ by
\[\phi_X(x) = \begin{cases} \phi(x) & \text{if } x \neq x_0, \\ \sum_{y \in Y} \phi(y) & \text{if } x = x_0.\end{cases}\]

We claim this is an $r$-Gorenstein labeling of $P_X$.
Indeed, suppose that $A \subseteq P_X$ is an upset of dimension 1.   If $x_0 \not\in A$, then
\[\sum_{x \in A} \phi_X(x) = \sum_{x \in A} \phi(x) = r\]
since $A$ is also an upset of $P$ with $\dim(A) = 1$.
In the case  where $x_0 \in A$, we let $B = A \cup Y$, which is an upset of $P$ with $\dim(B) = 1$.
It follows that
\begin{align*}
    \sum_{x \in A} \phi_X(x) &= \phi_X(x_0) + \sum_{x \in A \setminus \{x_0\}} \phi_X(x)\\
    &= \sum_{x \in Y} \phi(x) + \sum_{x \in A \setminus \{x_0\}} \phi(x)\\
    &= \sum_{x \in B} \phi(x)\\
    &= r
\end{align*}
and so $P_X$ has a $r$-Gorenstein labeling.

The same argument shows that the function $\phi_Y: P_Y \to \mathbb{Z}$   defined by 
\[\phi_Y(x) = \begin{cases} \phi(x) & \text{if } x \neq x_0 \\ \sum_{y \in X} \phi(y) & \text{if } x = x_0\end{cases}\]
is  an  $r$-Gorenstein labeling of $P_Y$.  The result now holds by induction  on the number of biconnected components of $P$.
\end{proof}

\begin{theorem}
 A poset $P$ has an $r$-Gorenstein labeling if and only if each biconnected component of $P$ has an $r$-Gorenstein labeling.
 \label{thm:bicon}
\end{theorem}
\begin{proof}
 The forward direction of the theorem is exactly Lemma~\ref{lem:bicomponent}.
 The backwards direction follows from Lemma~\ref{lem:gluing}, noting that $P$ can be built from the operation $\star$ applied to the biconnected components and cut vertices.
\end{proof}

\section{Posets with a minimum element or a maximum element}\label{sec:posetsWith0or1}
 
In this section we explore when a poset with a $\hat{0}$ or a $\hat{1}$ has a Gorenstein labeling.   We show that if $P$ has a $\hat{0}$ or a  $\hat{1}$, then $U_P$ is  Gorenstein if and only if it is $\mathbb{Q}$-Gorenstein (Theorem~\ref{thm:posetW0Or1GorIffQGor}).  Moreover, we will see that whether $U_P$ is Gorenstein, depends on the M\"obius function of $P$.  For posets $P$ with a $\hat{0}$ or a $\hat{1}$, we are able to give a  characterization of when   $U_P$ is Gorenstein (Theorem~\ref{thm:characterizationOfRankedPosetGor}). Using this characterization, we give an algorithm to check if $U_P$ is Gorenstein in this case (Algorithm~\ref{alg:RankedPosAlg}).  The characterization and the algorithm both make use of quotient posets.

A poset has a $\hat{0}$ if and only if its  order dual has a $\hat{1}$. By Proposition~\ref{prop:dual}, we know that $U_P$ is Gorenstein if and only if $U_{P^*}$ is Gorenstein.    Therefore we  only need to prove the results when either $P$ has a $\hat{0}$ or has a $\hat{1}$.  We believe the presentation of the proofs is simpler in the case that the poset has a $\hat{1}$ and thus we have chosen to present the proofs this way.  

Throughout this section, we will often need to consider downsets generated by a single element.  We will use the notation $\downset{x}$ to denote the downset generated by $x$. Similarly we will use $\downset{S}$ for the downset generated by the set $S$.
 
\subsection{$\mathbb{Q}$-Gorenstein implies Gorenstein}
 If $P$ has  a $\hat{1}$ then for any downset $A\neq P$, $\overline{A}$ is connected. Since any principal downset is connected, this forces any possible Gorenstein labeling to be completely determined by principal downsets.  This allows us to give a recursive definition for an $r$-Gorenstein labeling.

\begin{lemma}\label{lem:rGorUniqueLabel}
Let $P$ be a poset with a $\hat{1}$.  If $P$ has a $r$-Gorenstein labeling, then this labeling is given by
$$
\sum_{y \leq x} \phi(y) = r(\delta_{x,\hat{1}}-1)
$$
where $\delta_{x,\hat{1}}$ is the Kronecker delta function.

\end{lemma}
\begin{proof}
Since $P$ has a $\hat{1}$, if $A$ is a proper downset, then $\overline{A}$ is connected.   Now suppose that $A = \downset{x}$ with $x\neq\hat{1}$.  Then since $x$ is the unique maximum element of $A$,  $A$ is  connected.  Thus, $\dim(A)=1$.
 
Now $\downset{x} = P$ if and only if $x=\hat{1}$. It follows that from  Proposition~\ref{prop:downsetLabel} that
$$
\sum_{y\in \downset{x}} \phi(y)   =   r(\delta_{x,\hat{1}}-1)
$$
or equivalently, 
$$
     \sum_{y \leq  x} \phi(y)   =   r(\delta_{x,\hat{1}}-1).
$$
Thus, the result holds.     
\end{proof}

We are now ready to prove one of the main theorems in this section.

 \begin{theorem}\label{thm:posetW0Or1GorIffQGor}
 Let $P$ be a poset with a $\hat{0}$ or a $\hat{1}$.  Then  $U_P$ is Gorenstein if and only if $U_P$ is $\mathbb{Q}$-Gorenstein.
\end{theorem}
\begin{proof}
We prove the result when $P$ has a $\hat{1}$ which suffices by Proposition~\ref{prop:dual}. The forward direction is immediate given the definitions of Gorenstein and $\mathbb{Q}$-Gorenstein.

For the backwards direction, it suffices to show that  if $\phi$ is the unique $r$-Gorenstein labeling given in Lemma~\ref{lem:rGorUniqueLabel}, then $r\mid \phi(x)$ for all $x\in P$.  We do this by inducting on the maximum length of a chain from a minimal element to $x$.  If the length is 0, then $x$ must be a minimal element. So by Lemma~\ref{lem:rGorUniqueLabel}, $\phi(x)=-r$ and so the base case holds.

Now suppose that the maximum length of a  chain from $x$ to a minimal element of $P$ is $k$ and that $x\neq \hat{1}$. Then by Lemma~\ref{lem:rGorUniqueLabel},
$$
\phi(x) = -\sum_{y<x} \phi(y) -r.
$$
By the inductive hypothesis, we have that $r\mid \phi(y)$ for all $y<x$.  Thus, $r\mid \phi(x)$ in this case. Finally, by Lemma~\ref{lem:rGorUniqueLabel}, 
$$
\phi(\hat{1})  = -\displaystyle\sum_{y\neq \hat{1}} \phi(y).
$$
7We have already shown that  $r\mid \phi(y)$ for all $y\neq \hat{1}$ and so $r\mid \phi(\hat{1})$.  It follows that $\dfrac{1}{r}\phi(x) \in \mathbb{Z}$ for all $x$ and so   $\dfrac{1}{r}\phi(x)$ is a Gorenstein labeling of $P$.
\end{proof}

\subsection{Relationship with the M\"obius function}

Given the previous theorem, we will only focus on Gorenstein labelings (as opposed to $r$-Gorenstein labelings) in the remainder of this section.  In the next proposition, we   show that the labeling described in  Lemma~\ref{lem:rGorUniqueLabel} (when $r=1$) is related to the M\"obius function.    In this proposition and elsewhere in this section, we will be considering  a new poset, namely, the   poset  obtained by adding a minimum element $\hat{0}$ to $P$. We note that we add a $\hat{0}$ even if $P$ already has a minimum element.  The poset obtained by adding a $\hat{0}$ to $P$ will be denoted by $\widehat{P}$.  We now give the definition of the M\"obius function.

\begin{definition}
Let $P$ be a poset with a $\hat{0}$. The \emph{(one-variable) M\"obius function}, $\mu_P(x): P \to \mathbb{Z}$, is defined recursively by
$$
\sum_{y\leq x} \mu_P(y)  = \delta_{\hat{0},x}
$$
where $ \delta_{\hat{0},x}$ is Kronecker delta function.
\end{definition}

We will use the notation $\mu_{\widehat{P}}(x)$ for the M\"obius function of $x$ in $\widehat{P}$. See Figure~\ref{fig:MP} for examples of the M\"obius function.   From time to time, we  will need to consider M\"obius functions of  different posets. In these cases, we will use the subscripts  to distinguish them.

\begin{lemma}\label{lem:MobAndPhi}
Let $P$ be a poset with $\hat{1}$ and let $\phi$ be the labeling described in Lemma~\ref{lem:rGorUniqueLabel} when $r=1$.    Then for all $x\in P$,
$$
\phi(x) = 
\begin{cases}
\mu_{\widehat{P}}(x)& \mbox{if $x\neq \hat{1}$,}\\
  \mu_{\widehat{P}}(x) +1 & \mbox{if $x=\hat{1}$.}\\
\end{cases}
$$
\end{lemma}

\begin{proof}
We prove the result for $x\neq \hat{1}$, the argument for  $x=\hat{1}$ is  similar.

We induct on the maximum length of a chain from a minimal element of $P$ to $x$. If the length is 0, then $x$ is minimal.  Using the definition of the M\"obius function, we have that $\mu_{\widehat{P}}(x) =-1$.  This agrees with the value given in Lemma~\ref{lem:rGorUniqueLabel} (with $r=1$) and so the base case holds.  

Now suppose that the maximum length is more than 0. Then, by  Lemma~\ref{lem:rGorUniqueLabel} (with $r=1$), we have that 
$$
\sum_{y \leq x} \phi(y) = -1.
$$
Rearranging gives us that
$$
\phi(x) = -\sum_{y < x} \phi(y) -1.
$$
By the inductive hypothesis, $\phi(y) =\mu_{\widehat{P}}(y)$ for all $y<x$. Moreover, by definition, $\mu_{\widehat{P}}(\hat{0}) =1$. Thus the previous displayed equation becomes
$$
\phi(x) = -\sum_{y < x} \mu_{\widehat{P}}(y) - \mu_{\widehat{P}}(\hat{0}).
$$
Again using the definition of the M\"obius function, we see that the righthand side is exactly $\mu_{\widehat{P}}(x )$. Thus the result holds by induction.   
\end{proof}

Combining  Lemma~\ref{lem:rGorUniqueLabel} and Lemma~\ref{lem:MobAndPhi}, we get the following characterization of when $U_P$ is Gorenstein in terms of the M\"obius function.

\begin{theorem}\label{thm:mobGorensteinLab}
 Let $P$ be a poset with a $\hat{1}$.  If $\phi$ is a Gorenstein labeling of  $P$, then 
$$
\phi(x) =
\begin{cases}
\mu_{\widehat{P}}(x) &\mbox{if } x\neq \hat{1},\\
\mu_{\widehat{P}}(x)+1  &\mbox{if } x= \hat{1}.\\
\end{cases}
$$
Consequently, $U_P$ is Gorenstein if and only if 
$$
\sum_{x\in A} \mu_{\widehat{P}}(x) =-1
$$
for all  downsets $A$ of dimension 1.
\end{theorem}

\subsection{A characterization and algorithmic considerations}

As we will see when $P$ has a $\hat{1}$,  the minimal elements of $P$ and the elements covering them play a crucial role in determining if $U_P$ is Gorenstein.  Because of this, we introduce notation for these elements.  We will use  $M_P$ to denote the induced subposet of elements that are either minimal or cover only  minimal elements.  That is,
$$
M_P =\{ x\in P \mid \mbox{if } y\lessdot x, \mbox{ then $y$ is minimal}\}.
$$ 
See Figure~\ref{fig:MP} for two examples.  Note that in the example, $e\notin M_Q$.  This is because while $e$ does  cover $b$, it also covers $c$ and $d$ which are not minimal.

 \begin{figure} 
     \begin{tikzpicture}
     \begin{scope}[scale=2,rotate=180, xscale=-1]
     
      \node (1) at (0,0) {$\hat{1}$};
      
        \node  (a) at (-1,1) {$f$};
        \node  (b) at (0,1) {$g$};
        \node  (c) at (1,1) {$h$};
        \node (d) at (-1,2) {$\textcolor{blue}{c}$};
         \node (e) at (0,2) {$\textcolor{blue}{d}$};
          \node (f) at (1,2) {$\textcolor{blue}{e}$};
           \node (g) at (-.5,3) {$\textcolor{blue}{a}$};
            \node (h) at (.5,3) {$\textcolor{blue}{b}$};

            \node at (0-.25,0) {\textcolor{red}{$-1$}};
             \node   at (-1-.25,1) {\textcolor{red}{$0$}};
             \node    at (0-.25,1) {$\textcolor{red}{+1}$};
               \node   at (1-.25,1)  {$\textcolor{red}{0}$};
             \node  at (-1-.25,2) {$\textcolor{red}{0}$};
            \node  at (0-.25,2){$\textcolor{red}{+1}$};
            \node at (1-.25,2)   {$\textcolor{red}{0}$};
            \node  at (-.5-.25,3) {$\textcolor{red}{-1}$};
            \node at (.5-.25,3) {$\textcolor{red}{-1}$};
            
            \draw (1)--(a);
              \draw (1)--(b);
              \draw (1)--(c);
              
                \draw (e)--(a)--(d);
              \draw (f)--(b)--(d);
              \draw (e)--(c)--(f);
              \draw (d)--(g);
              \draw (g)--(e);
               \draw (e)--(h);
                \draw (h)--(f);
         
         \node at (0,3.5) {$P$} ;
         
              \end{scope}
  
        \begin{scope}[shift={(6,0)},scale=2,rotate=180, xscale=-1]
        
           \node (1) at (0,0) {$\hat{1}$};
        \node  (a) at (0,1) {$ e$};
        \node  (b) at (-1,2)   {$\textcolor{blue}{c}$};
        \node  (c) at (0,2)  {$\textcolor{blue}{d}$};
        \node  (d) at (-1,3)  {$\textcolor{blue}{a}$};
         \node  (e) at (1,3)  {$\textcolor{blue}{b}$};
         \draw (1)--(a)--(b)--(d)--(c)--(a)--(e);

          \node at (0-.25,0) {\textcolor{red}{$0$}};
                \node    at (0-.25,1) {$\textcolor{red}{+1}$};
              \node  at (-1-.25,2) {$\textcolor{red}{0}$};
            \node  at (0-.25,2){$\textcolor{red}{0}$};
             \node  at (1-.25,3) {$\textcolor{red}{-1}$};
            \node at (-1-.25,3) {$\textcolor{red}{-1}$};

               \node at (0,3.5) {$Q$} ;

         \end{scope}
     \end{tikzpicture}
     \caption{Posets $P$ and $Q$ with the elements of $M_P$ and $M_Q$ colored blue.  The values next to an element $x$ given in red correspond to $\mu_{\widehat{P}}(x)$.}\label{fig:MP}
\end{figure}
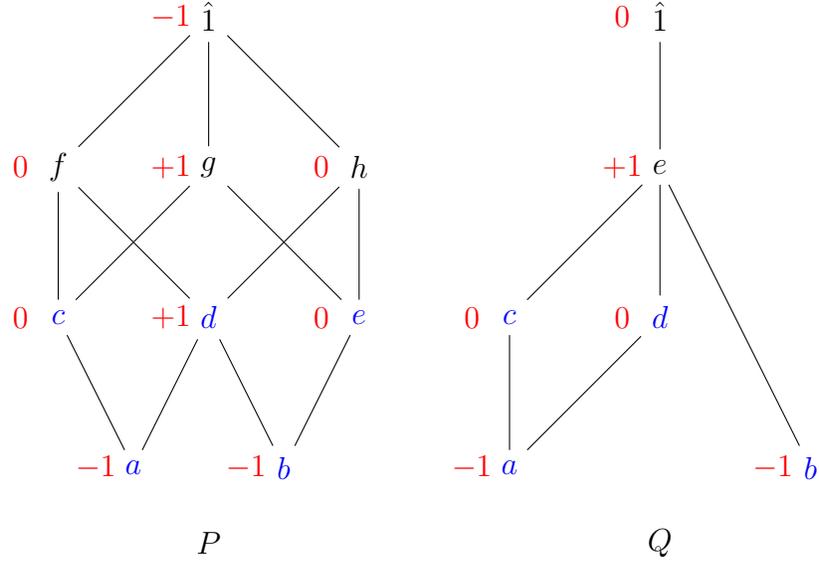

If $U_P$ is  Gorenstein, we must have that the sum of the M\"obius values of any connected downset contained in $M_P$ is $-1$.  Thus it is important for us to understand when this happens. As we show in the next lemma, this forces the connected downsets of $M_P$ to be trees.

\begin{lemma}\label{lem:mobOfTree}
If $C$ is a connected downset of $M_P$, then 
$$
\sum_{x\in C}\mu_{\widehat{P}}(x) = -1
$$
if and only if $C$ is a tree.
\end{lemma}
\begin{proof}
 A direct calculation using the definition of the M\"obius function shows that for $x\in M_P$,
$$
\mu_{\widehat{P}}(x) = 
\begin{cases}
 -1 & \mbox{if $x$ is minimal in $P$}\\
 \deg_{M_P}(x)-1 & \mbox{otherwise}
\end{cases}
$$
where $\deg_{M_P}(x)$ refers to the degree of $x$ considered as a vertex in the Hasse diagram of $M_P$.  Let $C'$ be the minimal elements of $C$ and let $C''$ be the maximal elements of $C$.  Then
\begin{align*}
    \sum_{x\in C} \mu_{\widehat{P}}(x) &= \sum_{x\in C'} \mu_{\widehat{P}}(x)  +\sum_{x\in C''}\mu_{\widehat{P}}(x)  \\
    & = \sum_{x\in C'} (-1) + \sum_{x\in C''} ( \deg_{M_P}(x)-1 )\\
    & = -|C'| +  \sum_{x\in C''}  \deg_{M_P}(x)- |C''|\\
    & = -|C|  +  \sum_{x\in C''}  \deg_{M_P}(x)
\end{align*}
Now $C$ is bipartite with partite sets $C'$ and $C''$.  It follows that 
$$
\sum_{x\in C''}  \deg_{M_P}(x)
$$
is the number of edges in $C$.  Combining  this fact with the above equations shows that 
$$
    \sum_{x\in C} \mu_{\widehat{P}}(x)  =  |E(C)| - |V(C)|
$$
where $E(C)$ is the set of edges of $C$ and $V(C)$ is the set of vertices in $C$.   Since $C$ is connected, $ |E(C)| - |V(C)| =-1$ if and only if $C$ is tree and so the result holds.
 \end{proof}

  Lemma~\ref{lem:mobOfTree} along with 
  Theorem~\ref{thm:mobGorensteinLab}, imply that $M_P$ must be acyclic for $U_P$ to be Gorenstein.  However, this is not enough to guarantee that $U_P$ is Gorenstein.  For example, consider the poset $P$ in Figure~\ref{fig:MP}. The sum of the M\"obius values of $M_P$ is $-1$  and the same is true for $\downset{g}$.   However, if we consider the downset $R\cup \downset{g}$,  where $R=\{a,b,d\}$ is the path in $M_P$ connecting $a$ and $b$, we see that the sum of the M\"obius values is $0$. However, it needs to be $-1$ for $U_P$ to be Gorenstein. It turns out  the reason this arises is because while $a$ and $b$ are in the same tree of $M_P$,  they are not in the same tree of $M_P\cap  \downset{g}$.  This allows one to add in the vertices of  this ``missing" $a-b$ path in  $M_P\cap  \downset{g}$ and keep the downset connected. However, by adding this path, we change the sum of the M\"obius values. It turns out that this  issue is not unique to this example.  To describe what is happening, we need a definition.  In the next definition and elsewhere, we will need to count the number of connected components in a poset. As before, we will use the notation $cc(P)$ for the number of connected components of the Hasse diagram of $P$.

 \begin{definition}
Let $P$ be a poset with a $\hat{1}$ and let $C$ be a connected downset of $P$.  We say $C$ satisfies the \emph{connected component condition} (or simply the \emph{cc condition}) if the number of connected components of $M_P$ that have nonempty intersection with $C$ is equal to $cc(C\cap M_P)$. We say $P$ satisfies the cc condition  if all of its connected downsets satisfy it. 
 \end{definition}

The downset $\downset{g}$ in the poset $P$ of Figure~\ref{fig:MP} does not satisfy the cc condition. Indeed,  the number of connected components of $M_P$  that have nonempty intersection with $\downset{g}$ is 1 since $M_P$ is connected. However, $\downset{g}\cap M_P$ has two connected components, namely $\{a,c\}$ and $\{b,e\}$. On the other hand, one can check that all connected downsets in $Q$ of Figure~\ref{fig:MP}  satisfy the cc condition. And so $Q$ satisfies the cc condition.


\begin{lemma}\label{lem:NecessaryCondForGor}
Let $P$ be a poset with a $\hat{1}$.  If $U_P$ is Gorenstein, then  the following hold.
\begin{enumerate}
    \item[(1)] $M_P$ is acyclic.
     \item[(2)] $P$ satisfies the cc condition.
     \end{enumerate}
\end{lemma}

\begin{proof}
First note that if  (1) did not hold, then there would be a connected downset containing a cycle. But then by Lemma~\ref{lem:mobOfTree}, the sum of the M\"obius values in this downset would not be $-1$, contradicting Theorem~\ref{thm:mobGorensteinLab}.

Now suppose that (2) did not hold and let $C$ be a connected downset that does not satisfy the cc condition.  Each connected component of $C\cap M_P$ must be contained in a connected component of $M_P$.  Thus, the number of components of $M_P$ that have nonempty intersection with $C$ is less than  $cc(C\cap M_P)$.  It follows that there is a path completely contained in $M_P$ that connects two connected components in $C\cap M_P$.  Let $C'$ be the downset generated by $C$ and this path.   Then $C'$ is a connected downset.  By construction, $C\setminus M_P =C'\setminus M_P$ and $cc(C\cap M_P) \neq cc(C'\cap M_P)$.  But then by Lemma~\ref{lem:mobOfTree}, $C$ and $C'$ have different sums of M\"obius values since each connected component of $C\cap M_P$ and $C'\cap M_P$ contributes $-1$ to the total sum.  This is impossible as $\dim(C) =1=\dim(C')$ and $U_P$ is Gorenstein.
\end{proof}

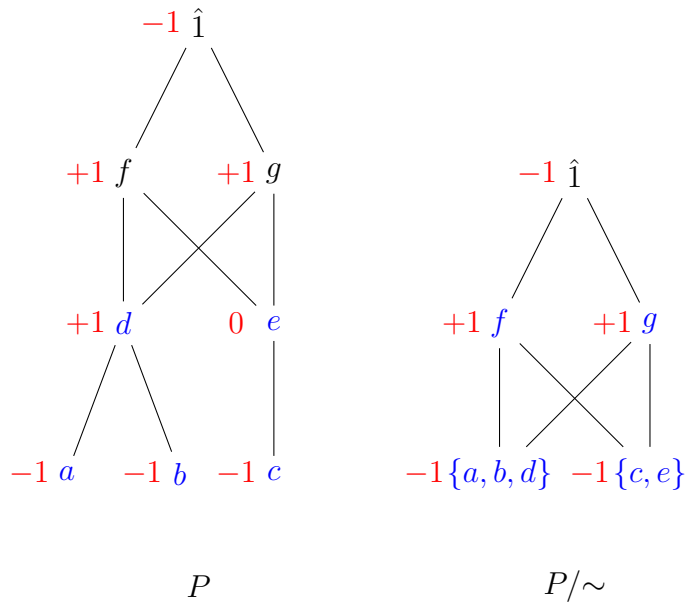
\begin{figure}
    \begin{tikzpicture}
     \node  (a) at (-1.75,0)  {$\textcolor{blue}{a}$};
          \node  (a') at (-.25,0)  {$\textcolor{blue}{b}$};

     \node (b) at (1,0)  {$\textcolor{blue}{c}$};
      \node  (c) at (-1,2)  {$\textcolor{blue}{d}$};
       \node (d) at (1,2)  {$\textcolor{blue}{e}$};
       \node  (e) at (-1,4)  {$f$};
       \node (f) at (1,4)  {$g$};
    \node (1) at (0,6)  {$\hat{1}$};
    
    \draw (a')--(c);
    \draw (a)--(c)--(e)--(1)--(f)--(d)--(b);
    \draw (c)--(f);
    \draw  (d)--(e);
    
         \node   at (-1.75-.5,0)  {$\textcolor{red}{-1}$};
          \node   at (-.25-.5,0){$\textcolor{red}{-1}$};
          
          \node   at (1-.5,0)  {$\textcolor{red}{-1}$};

        \node   at (-1-.5,2)  {$\textcolor{red}{+1}$};
          \node   at (1-.5,2)  {$\textcolor{red}{0}$};
      \node   at (-1-.5,4)  {$\textcolor{red}{+1}$};
          \node   at (1-.5,4)  {$\textcolor{red}{+1}$};
            \node   at (0-.5,6)   {$\textcolor{red}{-1}$};

             \node at (0,-1.5) {$P$} ;
    \begin{scope}[shift ={(5,-2)}]

       \node  (c) at (-1,2)  {$\textcolor{blue}{\{a,b,d\}}$};
       \node (d) at (1,2)  {$\textcolor{blue}{\{c,e\}}$};
       \node  (e) at (-1,4)  {$\textcolor{blue}{f}$};
       \node (f) at (1,4)  {$\textcolor{blue}{g}$};
    \node (1) at (0,6)  {$\hat{1}$};

    \draw (c)--(e)--(1)--(f)--(d);
    \draw (c)--(f);
    \draw  (d)--(e);

        \node   at (-1-1,2)  {$\textcolor{red}{-1}$};
          \node   at (1-.8,2)  {$\textcolor{red}{-1}$};
      \node   at (-1-.5,4)  {$\textcolor{red}{+1}$};
          \node   at (1-.5,4)  {$\textcolor{red}{+1}$};
            \node   at (0-.5,6)   {$\textcolor{red}{-1}$};

             \node at (0,.5) {$P/{\sim}$} ;
    \end{scope}

    \end{tikzpicture}
     \caption{Posets $P$ and $P/{\sim}$ with the elements of $M_P$ and $M_{P/{\sim}}$ colored blue.  The values next to an element $x$ given in red correspond to $\mu_{\widehat{P}}(x)$.}\label{fig:NecessCondNotSuffExample}
\end{figure}

While the two conditions in Lemma~\ref{lem:NecessaryCondForGor} are necessary for $U_P$ to be Gorenstein,  they are not sufficient.  To see why, consider the poset $P$ in Figure~\ref{fig:NecessCondNotSuffExample}.   Both conditions of  Lemma~\ref{lem:NecessaryCondForGor}  hold. However, $U_P$ is not Gorenstein since the sum of the M\"obius values of the downset $\{a,b,c,d,e,f,g\}$ is 0.   Even though the two conditions in Lemma~\ref{lem:NecessaryCondForGor} are not sufficient, we will see that when they hold, we can modify $P$ in a way that reduces the number of elements of $P$ and does not change the Gorensteinness of the  corresponding variety.  Before we explain this in detail, let us return to $P$ given in Figure~\ref{fig:NecessCondNotSuffExample}.  As noted before, $M_P$ is acyclic.  Consider a new poset obtained by collapsing the elements of each tree of $M_P$ to a single element.  That is, collapse the elements, $a,b$ and $d$ to a single element and $c$ and $e$ to a single element. This new poset we get is labeled as $P/{\sim}$ in  Figure~\ref{fig:NecessCondNotSuffExample}.    Note that $M_{P/{\sim}}$ is not acyclic and so by Lemma~\ref{lem:NecessaryCondForGor},   $U_{P/{\sim}}$ is not Gorenstein.  As we will see, the fact that $U_{P/{\sim}}$ is not Gorenstein can be used to explain why $U_P$ is not Gorenstein.

We now turn our attention to  providing a characterization and an algorithm to determine when a  poset with a $\hat{0}$ or $\hat{1}$ has variety that is Gorenstein. This will formalize the ideas we discussed in the previous paragraph.  We have already discussed one of the main tools in this characterization, the M\"obius function.  The other main tool is quotient posets, which we now examine.

\begin{definition}\label{def:quoPoset}
Let $P$ be a poset and let $\sim$ be an equivalence relation on $P$. The \emph{quotient}, $P/{\sim}$, is the set of equivalence classes with the binary
relation $\leq$  given by $X\leq Y$ in $P/{\sim}$ if and only if $x\leq y$  in $P$ for some $x\in X$ and $y\in Y$.  If $P/{\sim}$ is a poset, we call $P/{\sim}$ the \emph{quotient poset}.
\end{definition}

\begin{remark}
Whenever we are using quotient posets, we will use uppercase letters for elements of the quotient poset and lowercase letters for the original poset.   We will use $\leq$ for both orders, despite the fact that they are not the same. 
\end{remark}

In general, the quotient  described in Definition~\ref{def:quoPoset}  need not be a poset.  For example, take a chain of length 2 and identify the bottom and top elements.  The corresponding $\leq$ relation defined for the quotient is reflexive and is  transitive, but it is not antisymmetric.   However, the relation that we  use here  will  produce a poset.

\begin{definition}
Let $P$ be a poset such that $M_P$ is acyclic.  The \emph{tree relation} $\sim$ is the relation such that  $x\sim y$ if and only if $x=y$ or $x$ and $y$ are in the same connected component of $M_P$.  
\end{definition}
See Figure~\ref{fig:NecessCondNotSuffExample} for an example of  a poset $P$ and the quotient poset $P/
{\sim}$ where $\sim$ is the tree relation.  We will now show the tree relation produces a quotient that is indeed a poset.   Note that we assume that $M_P$ is acyclic. Although this is not necessary for the quotient to be a poset, whenever we apply the relation, it will always be acyclic.  This also explains the use of the word ``tree" in ``tree relation".

\begin{lemma}
Let $P$ be a poset such that $M_P$ is acyclic. Let $\sim$ be the tree relation.  Then
$P/{\sim}$ is a poset. 
 
 \end{lemma}
\begin{proof}
First note that it is trivial to verify that $ \leq$ is reflexive.  

To show that $\leq$ is antisymmetric, suppose on the contrary that there exists  $X,Y \in P/{\sim}$ such that  $X\leq Y$ and $Y\leq X$, but $X\neq Y$.     Since $X\leq Y$, there is a $x\in X$ and a $y\in Y$ such that $x\leq y$.  Since $X\neq Y$, we know that $x$ and $y$ are not in the same equivalence class.   We claim that this means that $y\notin M_P$.  Indeed, if $y\in M_P$, then $x\leq y$ would imply that $x\in M_P$.  However, if $x,y\in M_P$ and $x\leq y$, then $x$ and $y$ are in the same tree of $M_P$ and this would imply that $X=Y$.    Now since $Y\leq X$,  we have that there is an element in $Y$ below an element of $X$.  Since $y\notin M_P$, $Y$ only has one element and this element must be $y$. So there exists a $x'\in X$ such that $y\leq x'$.  The fact that $y\notin M_P$ and $y\leq x'$ implies that $x'\notin M_P$.  But then this means that $X$ must contain a single element and so $x=x'$. This together with $x\leq y\leq x'$ implies that $x=y$ and so $X=Y$, a contradiction.

Finally, we show that $\leq$ is transitive.  Suppose that $X\leq Y$ and $Y\leq Z$.  If $X=Y$ or $Y=Z$, we immediately get that $X\leq Z$. So we may assume that $X<Y$ and $Y<Z$.  Since $X<Y$, there is a $x\in X$, $y\in Y$ such that $x<y$.  Using the same argument from the previous paragraph, we can see that $y\notin M_P$.  This implies that $Y=\{y\}$. So $Y<Z$ implies that there is a $z\in Z$ such that $y<z$.  So we have $x<y<z$, implying that $x<z$. Thus $X<Z$ and so $\leq$ is transitive.
\end{proof}

By Theorem~\ref{thm:mobGorensteinLab} whether or not a poset with a $\hat{1}$ corresponds to a Gorenstein variety depends on the M\"obius function.  Because of this, it is prudent to understand how the M\"obius function behaves when using the tree relation.  If we return to our example poset  $P$ in Figure~\ref{fig:NecessCondNotSuffExample}, we see that when we took the quotient, none of the M\"obius values of the elements outside of $M_P$ changed.  This is because  when the trees collapsed to a single element, the M\"obius value (in the quotient) of the single element is $-1$ which is the same as the sum of the M\"obius values in the tree before it was collapsed.   As we see in the next lemma,  this  property is not a unique to our example.

\begin{lemma}\label{lem:MobInQuot}
Let $P$ be a poset with a $\hat{1}$.  Suppose that 
\begin{enumerate}
     \item[(1)] $M_P$ is acyclic.
     \item[(2)] $P$ satisfies the cc condition.
\end{enumerate}
 Let $\sim$ be the tree relation on $P$.  Then for all $x\in P\setminus M_P$
$$
\mu_{\widehat{P}}(x) = \mu_{\widehat{P/{\sim}}}(X)
$$
where  $X$ is the equivalence class containing $x$.
\end{lemma}

\begin{proof}
We induct on the number of elements strictly below $x$ not in $M_P$.  Suppose that this number is 0.  In other words, $x$ is only above elements of $M_P$.  Let $T_1,T_2,\dots, T_k$ be the connected components of $\downset{x} \cap M_P$.  By Lemma~\ref{lem:mobOfTree} the sum of M\"obius values of a tree is $-1$ and so
$$
\mu_{\widehat{P}}(x) = k-1.
$$
 
 Since $x\notin M_P$, we know that $x$ is in its own equivalence class, $X$.   In $P/{\sim}$,   each tree $T_i$ has been collapsed to a single element and these elements are minimal in $P/{\sim}$.      Moreover because of the assumption (2) applied to the connected downset $\downset{x}$ and the definition of a quotient poset, $\downset{X}\setminus X = \{T_1,T_2,\dots, T_k\}$. Since the $T_i$'s are minimal in $P/{\sim}$, we have that
$$
\mu_{\widehat{P/{\sim}}}(T_i) =-1.
$$ 
and so 
$$
\mu_{\widehat{P/{\sim}}}(X)=k-1.
$$
Thus, in this case, $\mu_{\widehat{P}}(x) = \mu_{\widehat{P/{\sim}}}(X)$   and so the base case holds.\\

Now suppose that the number of elements below $x$ outside of $M_P$ is nonzero. As before, let $X=\{x\}$.  By the inductive hypothesis, if $Y<X$ and $Y$ contains no elements of $M_P$, then 
$$
 \mu_{\widehat{P}}(y) =\mu_{\widehat{P/{\sim}}}(Y)
$$
where $Y=\{y\}$. Moreover, by the definition of the quotient poset, if $Y = \{y\}$ then $
Y<X$ if and only if $y<x$. Using these facts, the definition of the M\"obius function, and condition (2), we see that
$$
 \mu_{\widehat{P}}(x) = \mu_{\widehat{P/{\sim}}}(X)
$$
and so the result holds by induction.
\end{proof}

By Lemma~\ref{lem:mobOfTree}, we know that when $T$ is a tree, the sum of the M\"obius values is $-1$. This agrees with  the M\"obius value of the tree (as a single element) in the $P/{\sim}$.  In other words, the sum of the M\"obius values in that equivalence class is the M\"obius value of the equivalence class in the quotient. Every other equivalence class only contains a single element. By Lemma~\ref{lem:MobInQuot}, under the conditions given in that lemma, the same summation property holds for these other elements in the quotient. Thus, we have the following.

\begin{corollary}\label{cor:sumOfMobVal}
Let $P$ be a poset with a $\hat{1}$.  Suppose that 
\begin{enumerate}
   \item[(1)] $M_P$ is acyclic.
     \item[(2)] $P$ satisfies the cc condition.
\end{enumerate}
 Let $\sim$ be the tree relation on $P$.  Then for all $X\in P/{\sim}$,
$$
\mu_{\widehat{P/{\sim}}}(X) = \sum_{x\in X} \mu_{\widehat{P}}(x).
$$
 \end{corollary}

Now that we have an understanding of how the M\"obius function behaves when taking quotients, we move to understanding how connected downsets behave. As we will see in the following technical lemma, there is a mapping between connected downsets of $P$ and $P/{\sim}$ that preserves the sum of the M\"obius values.

\begin{lemma}\label{lem:connectedDownsets}
Let $P$ be a poset with a $\hat{1}$.  Suppose that 
\begin{enumerate}
      \item[(1)] $M_P$ is acyclic.
     \item[(2)] $P$ satisfies the cc condition.
\end{enumerate}
Let $\sim$ be the tree relation on $P$  and let $C(P)$ and $C(P/{\sim})$ be the set of connected downsets of $P$ and $P/{\sim}$.  Define a map $\psi: C(P) \rightarrow C(P/{\sim})$ by $\psi(A) = \{X \in P/{\sim} \mid X\cap A \neq \emptyset\}$.   Then the following hold.
\begin{enumerate}
    \item[(a)] $\psi$ is a well-defined surjection.
    \item[(b)] If $A,B  \in \psi^{-1}(D)$, then  
    $$
    \displaystyle \sum_{x\in A} \mu_{\widehat{P}}(x) = \displaystyle \sum_{x\in B} \mu_{\widehat{P}}(x).
    $$
    \item[(c)] For all $A\in C(P)$, 
    $$\displaystyle  \sum_{x\in A} \mu_{\widehat{P}}(x) =  \sum_{X\in \psi(A)}\mu_{\widehat{P/{\sim}}}(X).
    $$
\end{enumerate}
\end{lemma}

\begin{proof}
First, let us prove (a).  We start by  showing $\psi$ is well-defined.  Suppose that $A$ is a connected downset of $P$.   Let $X,Y \in \psi(A)$ and let $x\in X\cap A$, $y\in Y\cap A$.  Since $x,y\in A$ and $A$ is connected, there exists an $x-y$ path in $A$, say $w_0w_1w_2\dots w_k w_{k+1}$ where $x=w_0$ and $y=w_{k+1}$.  Let $W_i$ be the equivalence class containing $w_i$. Since $w_i\in A$, $W_i\in \psi(A)$.  Since $w_i$ is comparable to $w_{i+1}$ in $P$, $W_i$ and $W_{i+1}$ are comparable or equal in $P/{\sim}$.  In any case, after removing repeated vertices, we see that $XW_1W_2\dots W_kY$ is  a path  in $P/{\sim}$ whose vertices are in $\psi(A)$. It follows that $\psi(A)$ is connected.

Next, let us show that $\psi(A)$ is a downset.  Suppose that $Y\in \psi(A)$ and let $X<Y$ in $P/{\sim}$.  Then since $X\lneq Y$, it must be the case that $Y$ is not minimal in $P/{\sim}$.  Then by definition of the tree relation, $Y=\{y\}$ for some $y\in P$.  Since $Y\in \psi(A)$ and $Y$ only has one element, this element must be in $A$.  That is, $y\in A$. Now $X<Y$ implies that there is  a $x\in X$ below an element of $Y$. Since $Y$ only has one element, we have  that $x<y$.  Since $A$ is a downset, $x\in A$.  Thus, $X\in \psi(A)$ and so $\psi(A)$ is a downset. We conclude that $\psi$ is well-defined.

Let us now show that $\psi$ is a surjection.  Let $D$ be a connected downset of $P/{\sim}$.  Let $A = \{a\in P \mid a\in Z \mbox{ for some } Z\in D\}$.  In other words, $A$ is the union of the equivalence classes contained in $D$.  Clearly, if $A\in C(P)$, then $\psi(A)=D$, so it suffices to show $A$ is connected and a downset.  First, we show $A$ is connected.  Let $x,y\in A$ and let $X,Y$ be the corresponding equivalence classes containing $x$ and $y$.  Then  we have that $X,Y\in D$.  Since $D$ is connected, there exists a path $W_0W_1W_2\dots W_kW_{k+1}$ in $D$ where $X=W_0$ and $Y= W_{k+1}$.   For each $i$, if $W_i$ and $W_{i+1}$ are both not minimal in $P/{\sim}$, then $W_i =\{w_i\}$ and $W_{i+1} = \{w_{i+1}\}$ for some elements $w_i$ and $w_{i+1}$. The fact that $W_i$ and $W_{i+1}$ are comparable in $P/{\sim}$ implies that $w_i$ and $w_{i+1}$ are comparable in $P$. Thus there is a path from $w_i$ to $w_{i+1}$. If $W_j$ is minimal in $P/{\sim}$, then $W_j$ corresponds to a tree in $M_P$.  Moreover, the fact that $W_{j-1}W_jW_{j+1}$ is a path in $D$ implies that $W_{j-1}$ and $W_{j+1}$ are not minimal.  It follows that $W_{j-1} = \{w_{j-1}\}$ and $W_{j+1} = \{w_{j+1}\}$ for some elements $w_{j-1}$ and $w_{j+1}$.  Since $W_{j-1}$ and $W_{j}$ are comparable in $P/{\sim}$, we know that there is an element $w\in W_j$ comparable to $w_{j-1}$.  Similarly, there is an element $w'\in W_j$ comparable to $w_{j+1}$.   Now all the elements of $W_j$ are in $A$ and since $W_j$ is  a tree and thus connected, there is a path from $w$ to $w'$ in $A$.  Thus, there is a path from $w_{j-1}$ to $w_{j+1}$ in $A$. By applying these ideas to the elements in  $W_0W_1W_2\dots W_kW_{k+1}$, we see that we can find a path from $x$ to $y$ in $A$.  Thus $A$ is connected.

Next, we show that $A$ is a downset.  Let $y\in A$ and let $x<y$ with $x\in X, y \in Y$.    Then $Y\in D$. Since $x<y$ in $P$, $X\leq Y$ in $P/{\sim}$.  Since $D$ is a downset this implies that $X\in D$.  It follows that $x\in A$ and so $A$ is a downset.  Thus  $\psi$ is a surjection and so (a) holds.

Now we prove (b).  Let $A,B \in \psi^{-1}(D)$. We claim that $A$ and $B$ can only differ in values of $M_P$.  To see why, suppose that there is an $x\notin M_P$ such that $x\in A$.    Let $X$ be the equivalence class containing $x$.  Then $X\in \psi(A)$ and so $\psi(A) =\psi(B)$, implies that  $X\in \psi(B)$.  Since  $x\notin M_P$, $X=\{x\}$ and so $x$ must be in $B$.


We now break into two cases depending on if $A\subseteq M_P$ or not.  If $A\subseteq M_P$, then since $A$ is connected, $A$ is a tree completely contained in $M_P$ and so by Lemma~\ref{lem:mobOfTree} the sum of the M\"obius values of $A$ is $-1$. Since the $A$ and $B$ agree except possibly for elements of $M_P$ and since $B$ is connected,  the sum of the M\"obius values of $B$ is $-1$ as well.  Thus (b) holds in this case.

Now suppose that $A$ is not contained in $M_P$ and suppose that there are $k$ connected components of $M_P$ that have nontrivial intersection with  $A$. Since $A$ is connected, for each such connected component, there is an element of $A\setminus M_P$ above some element of the component.  Since $A$ and $B$ agree except possibly for elements of $M_P$, we see that there must also be $k$ connected components of $M_P$ that intersect nontrivially with $B$.  By the cc condition, $cc(A\cap M_P) = k = cc(B\cap M_P)$. By Lemma~\ref{lem:mobOfTree} each connected component of $A\cap M_P$ and $B\cap M_P$ contributes $-1$ to the total sum of M\"obius values for $A$ and $B$. Since the two sets agree except possibly at $M_P$, we see that (b)  holds in this case too.

 Lastly let us prove (c). By (b), the sum of M\"obius values of any preimage of $\psi(A)$ is the same. Thus it suffices to show the result when $A = \{a\in P \mid a\in X \mbox{ for some } X\in \psi(A)\}$.  Then $A$ contains all the elements of the equivalence classes of $\psi(A)$.  Thus,
$$
 \sum_{x\in A} \mu_{\widehat{P}}(x)  = \sum_{X\in \psi(A)} \sum_{x\in X}    \mu_{\widehat{P}}(x) =   \sum_{X\in \psi(A)}\mu_{\widehat{P/{\sim}}}(X)
 $$
where the last equality holds by Corollary~\ref{cor:sumOfMobVal}.
\end{proof}

 The surjection $\psi$ described in Lemma~\ref{lem:connectedDownsets}, provides a mapping between connected downsets of $P$ and $P/{\sim}$ that preserves the sum of the M\"obius values.  Thus we immediately get the following.

 \begin{lemma}\label{lem:PGorIffPsimGor}
Let $P$ be a poset with a $\hat{1}$.  Suppose that 
\begin{enumerate}
    \item[(1)] $M_P$ is acyclic.
     \item[(2)] $P$ satisfies the cc condition.
\end{enumerate}
 Let $\sim$ be the tree relation on $P$.   Then $U_P$ is Gorenstein if and only if $U_{P/{\sim}}$ is Gorenstein.
\end{lemma}

We will  use the previous lemma to provide a characterization and an algorithm to determine when $U_P$ is Gorenstein provided $P$ has a $\hat{0}$ or $\hat{1}$. However, before we do that, we need to further discuss the cc condition.  Clearly we do not want to actually check all connected downsets for the cc condition, since if we already knew all the connected downsets, we could just sum the M\"obius values in these downsets to determine  if $U_P$ is Gorenstein.   Instead, we will show the cc condition is equivalent to a condition that, when the number of connected components of $M_P$ is small compared to the number of elements of $P$, can be checked efficiently.

\begin{definition}\label{def:treeDownsetCond}
Let $P$ be a poset with a $\hat{1}$ and suppose that $M_P$ is acyclic and contains  $k$ connected components.  For all sets $S \subseteq P \setminus M_P$ such that $1\leq |S| \leq k$, let $T_1,T_2,\dots, T_m$ be the connected components of $M_P$ such that $\downset{S}\cap T_i\neq \emptyset$.   Let $A_j = \downset{S,T_1,T_2,\dots, T_{j-1}, T_{j+1},\dots, T_m}$. We say that $P$ satisfies the tree downset condition if whenever $A_j$ is connected we have that $cc(A_j\cap M_P)=m$.
\end{definition}

Clearly, the cc condition implies the tree downset condition.  However, as  we will see, the tree downset condition is actually equivalent to it.   This equivalence means that if $M_P$ has $k$ connected components, we only need to look at downsets generated by   $k$ elements and the connected components of $M_P$ below them. This can significantly improve the efficiency of checking the cc condition.  For example, if $M_P$ is connected, the tree downset condition says we need to only check principal downsets for the cc condition.

In the proof of the next lemma we will need some new terminology.  Given a path in the Hasse diagram of $P$, $v_1v_2\dots v_\ell$,   we say $v_i$ is a \emph{peak} if $v_i>v_{i-1}, v_{i+1}$.    We say that  a $v_j$ is a    \emph{valley} if $v_j<v_{j-1}, v_{v_j+1}$.   See Figure~\ref{fig:PeaksAndValleys} for an example of peaks and valleys.

\begin{lemma}\label{lem:ccConditionEqTotdCondition}
Let $P$ be a poset with a $\hat{1}$ such that $M_P$ is acyclic. Then $P$ satisfies the cc condition if and only if  $P$ satisfies the tree downset condition.
\end{lemma}

\begin{proof}
First, note that the forward direction is immediate  given the definition of the cc condition and the tree downset condition.   Thus, we may now focus on the backwards direction.  Suppose that $P$ satisfies the tree downset condition, but there is a connected downset $C$  where the number of connected components of $M_P$ that intersect nontrivially with $C$ is not $cc(C\cap M_P)$.  Then it must be the case that there are elements $a,b\in C\cap M_P$ which are connected in $M_P$, but not connected in $C\cap M_P$.

 Since $C$ is connected, we know that there is an $a-b$ path in $C$.  
 Among all such $a-b$ paths, let $Q$ be a path which uses the minimal number of peaks outside of $M_P$.  We claim that $Q$ has at most $k-1$ peaks that are not in $M_P$.  To see why, suppose this was not the case and let $w_1,w_2,\dots, w_k$ be the first $k$ peaks of $Q$ not in $M_P$.

\begin{figure}
    \begin{tikzpicture}[scale=.8]
    
    \draw[]  (-2,2)- -(-1.5,1.5) ;
        \draw[]  (-2,2)--(-2.5,1.5) ;

\node[circle,  fill=blue,  inner sep=0, minimum size= 7 pt] (a) at (-2,2) {};
    \node at (-2,2.5) {$u$};
    \node[circle,  fill=black,  inner sep=0, minimum size= 7 pt] (a) at (-1.5,1.5) {};
        \node[circle,  fill=black,  inner sep=0, minimum size= 7 pt] (a) at (-2.5,1.5) {};

    \node[circle,  fill=red,  inner sep=0, minimum size= 7 pt] (a) at (0,0) {};

    \draw[dashed]  (-2,2)--(0,0);
       \node[circle,  fill=blue,  inner sep=0, minimum size= 7 pt] (a) at (2,2) {};
        \node[circle,  fill=black,  inner sep=0, minimum size= 7 pt] (a) at (1.5,1.5) {};
    \node[circle,  fill=black,  inner sep=0, minimum size= 7 pt] (a) at (2.5,1.5) {};
    \node at (2,2.5) {$v$};

            \draw[]  (2,2)--(1.5,1.5) ;
            \draw[]  (2,2)--(2.5,1.5) ;

      \draw[dashed]  (2,2)--(0,0);
    \draw[dashed] (-3, 1)--(-2,2); 
     \draw[dashed] (3, 1)--(2,2); 
    
    \draw[dotted, very thick] (-3.5,-.5) -- (3.5,-.5);
    
    \node at (3.25,-1) {$M_P$};
    \node at (3.25,0) {$P\setminus M_P$};

    \begin{scope}[shift = {(0,-4)}]
        \node[circle,  fill=blue,  inner sep=0, minimum size= 7 pt] (a) at (-2-4,2) {};
        \node[circle,  fill=black,  inner sep=0, minimum size= 7 pt] (a) at (-1.5-4,1.5) {};
        \node[circle,  fill=black,  inner sep=0, minimum size= 7 pt] (a) at (-2.5-4,1.5) {};
    \node at (-2-4,2.5) {$u$};

        \draw[]  (-2-4,2)--(-1.5-4,1.5) ;
        \draw[]  (-2-4,2)--(-2.5-4,1.5) ;
        
        \draw[dashed](-2.5-4-.5,1.5-.5)--(-2.5-4,1.5);
        
        \node[circle,  fill=red,  inner sep=0, minimum size= 7 pt] (a) at (-2.5,-1.5) {};
        \node[circle,  fill=blue,  inner sep=0, minimum size= 7 pt] (a) at (-2,-1) {};

        \draw[dashed](-1.5-4,1.5)--(-2.5,-1.5) ;
         \draw[](-2,-1)--(-2.5,-1.5) ;

        \node[circle,  fill=black,  inner sep=0, minimum size= 7 pt] (a) at (-1.5,-1.5) {};
        \draw[](-1.5,-1.5)--(-2,-1) ;

        \begin{scope}[shift={(1,0)}]
            
        \node[circle,  fill=red,  inner sep=0, minimum size= 7 pt] (a) at (-2.5,-1.5) {};
        \node[circle,  fill=blue,  inner sep=0, minimum size= 7 pt] (a) at (-2,-1) {};

        \draw[](-2,-1)--(-2.5,-1.5) ;

        \node[circle,  fill=black,  inner sep=0, minimum size= 7 pt] (a) at (-1.5,-1.5) {};
        \draw[](-1.5,-1.5)--(-2,-1) ;
        \end{scope}

          \begin{scope}[shift={(2,0)}]
            
        \node[circle,  fill=red,  inner sep=0, minimum size= 7 pt] (a) at (-2.5,-1.5) {};
        \node[circle,  fill=blue,  inner sep=0, minimum size= 7 pt] (a) at (-2,-1) {};

        \draw[](-2,-1)--(-2.5,-1.5) ;

        \node[circle,  fill=black,  inner sep=0, minimum size= 7 pt] (a) at (-1.5,-1.5) {};
        \draw[](-1.5,-1.5)--(-2,-1) ;
        \end{scope}

          \begin{scope}[shift={(3,0)}]
            
        \node[circle,  fill=red,  inner sep=0, minimum size= 7 pt] (a) at (-2.5,-1.5) {};
        \node[circle,  fill=blue,  inner sep=0, minimum size= 7 pt] (a) at (-2,-1) {};

        \draw[](-2,-1)--(-2.5,-1.5) ;

        \node[circle,  fill=black,  inner sep=0, minimum size= 7 pt] (a) at (-1.5,-1.5) {};
        \draw[](-1.5,-1.5)--(-2,-1) ;
        \end{scope}

          \begin{scope}[shift={(4,0)}]
            
        \node[circle,  fill=red,  inner sep=0, minimum size= 7 pt] (a) at (-2.5,-1.5) {};
        \node[circle,  fill=blue,  inner sep=0, minimum size= 7 pt] (a) at (-2,-1) {};

        \draw[](-2,-1)--(-2.5,-1.5) ;

        \node[circle,  fill=red,  inner sep=0, minimum size= 7 pt] (a) at (-1.5,-1.5) {};
        \draw[](-1.5,-1.5)--(-2,-1) ;
        \end{scope}

        \begin{scope}[shift = {(11.5,0)}]
          \node[circle,  fill=blue,  inner sep=0, minimum size= 7 pt] (a) at (-2-4,2) {};
        \node[circle,  fill=black,  inner sep=0, minimum size= 7 pt] (a) at (-1.5-4,1.5) {};
        \node[circle,  fill=black,  inner sep=0, minimum size= 7 pt] (a) at (-2.5-4,1.5) {};

    \node at(-2-4,2.5) {$v$};

        \draw[]  (-2-4,2)--(-1.5-4,1.5) ;
        \draw[]  (-2-4,2)--(-2.5-4,1.5) ;
        
        \draw[dashed](-1.5-4,1.5) --(-1-4,1);
        
        \draw[dashed](-2.5-4,1.5)--(3-12,-1.5);
          \end{scope}

    \draw[dotted, very thick] (-6,-.5) -- (6,-.5);
    
    \node at (5.25,-1) {$M_P$};
    \node at (5.25,0) {$P\setminus M_P$};

    \node at (0,-2) {$T$};

    \end{scope}

    \end{tikzpicture}
    \caption{Peaks are given  in blue and valleys are given in red.}
    \label{fig:PeaksAndValleys}
\end{figure}
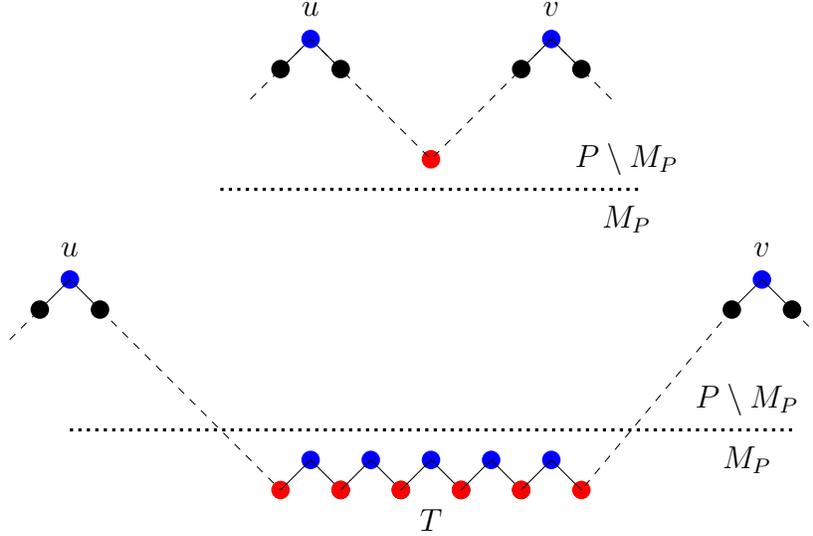

For each $i$, let $T_i$ be a tree of $M_P$ that has an element below the left valley of the peak $w_i$.  Moreover, let $T_{k+1}$ be a tree of $M_P$ that has an element  below the right valley of $w_k$. Since $M_P$ has $k$ trees, the sequence $T_1,T_2,\dots, T_k, T_{k+1}$ has a repeated element.  Call this repeated element $T$.   By definition, $T$ is below a left valley of a peak in $\{w_1,w_2,\dots, w_k\}$.  Moreover, we claim that $T$ is also below a right valley  of a peak in $\{w_1,w_2,\dots, w_k\}$.  If $T=T_{k+1}$, this is true by definition of $T_{k+1}$.  If not, then $T$ is below  a left valley of a peak say $v$. We claim that $T$ must be below  the right valley of the peak immediately preceding $v$ in the list $w_1,w_2,\dots, w_k$.  Call this peak $u$.   Clearly if the left valley of $v$ is the right valley of $u$ this is true. If not, it is because  the left valley of $v$ and the right valley of $u$ are in $M_P$ and the portion of the path between these valleys stay in $M_P$.  But then this means that both these valleys are in $T$.  See Figure~\ref{fig:PeaksAndValleys} for a visualization. We conclude that $T$ is below the right valley of some peak in $w_1,w_2,\dots, w_k$.  

Let $w_i$ and $w_j$ be such that $T$ is below the left valley of $w_i$ and $T$ is below the right valley of $w_j$ and such that $w_i$ and $w_j$ chosen so that they are closest in the list $w_1,w_2,\dots, w_k, w_{k+1}$ with this property.   Let $A$ be downset generated by $\{w_i,w_{i+1},\dots, w_j\}$ together with all the trees of $M_P$  below an element of  $\{w_i,w_{i+1},\dots, w_j\}$ except for $T$. We will show that all the peaks (including those in $M_P$) of the portion of $Q$ between $w_i$ and $w_j$ are in $A$.   Since the path between $w_i$ and $w_j$ is connected, the downset the path generates is connected.  Moreover since the path is generated by its peaks and since the additional trees added to $A$ are below elements of the path and the trees are connected themselves, we see that $A$ will be connected if we can show all the peaks are in $A$.

By definition, all the peaks outside of $M_P$ on the portion of $Q$ from $w_i$ to $w_j$ are in $A$.  Now suppose that $p$ is a peak in $M_P$ in the portion of $Q$ between $w_i$ and $w_j$.  Any time the path goes out of $M_P$, it create a peak outside of $M_P$ to go back down into $M_P$.  Since the trees of $M_P$ are connected components this means that once the path goes into $M_P$ it stays in the same tree of $M_P$ until the path leaves $M_P$.  Thus any peak in $M_P$ is in a tree that was below a  peak outside of $M_P$.  Since all the trees except for $T$ are completely contained in $A$, as long as these peaks in $M_P$ are not in $T$, they are in $A$.   Because  we chose $w_i$ and $w_j$ to be as close as possible, the peaks in $M_P$ on this portion of $Q$ cannot be in $T$.  Thus, all the peaks of $Q$ between $w_i$ and $w_j$ are in $A$ and so $A$ is connected.  

Let $t_1$ be an element of $T$ that is below the left valley of $w_i$ and let $t_2$ be an element of $T$ that is below the right valley of $w_j$.   $A$ is generated by no more than $k$ elements together with all but one tree of $M_P$ below it.  We claim that there is a $t_1-t_2$ path completely contained in $\downset{w_i,w_{i+1},\dots, w_j} \cap M_P$.  If this was not the case,  then this would contradict the tree downset condition. Since $w_i,w_{i+1},\dots, w_j \in C$, this path is completely contained in $C\cap M_P$.  However, we claim this allows us to modify $Q$ to get a new path from $a$ to $b$ that has less peaks outside of $M_P$. Let $Q'$ be the path from $a$ to $b$ which follows along $Q$ until we reach the left valley of $w_i$, then travels down to $t_1$, takes the path in $M_P$ from $t_1$ to $t_2$, then goes up to the right valley of $w_j$.  Note that this new path $Q$ does not contain any of the peaks in  $w_i,w_{i+1}, \dots, w_{j}$. When we descend from the left valley of $w_i$ to $t_1$ we do not create any new peaks.   Similarly we do not create new peaks when ascending from $t_2$ to the right valley of $w_j$.  We may create new peaks between $t_1$ and $t_2$, but these peaks are wholly contained in $M_P$.  As a result, $Q'$ has less peaks outside of $M_P$ than $Q$ does, contradicting the minimality of peaks outside of $M_P$ of $Q$.   We conclude that $Q$ has at most $k-1$ peaks. 

To finish, apply the same reasoning we did before to $Q$. Since $Q$ has at most  $k-1$ peaks outside of $M_P$, wee see that this allows us to find an $a-b$ path  in $C\cap M_P$, contradicting our original assumption.
\end{proof}

In the next theorem, we  provide a characterization of when the variety associated to a poset with a $\hat{1}$ is Gorenstein.  To do this, we  need some new notation.  Let $\sim$ be the tree relation on $P$.  We  define $P/{\sim}^k$ recursively as
$$
P/{\sim}^k  :=
\begin{cases}
 P &\mbox{ if } k=0,\\
 (P/{\sim}^{k-1})/{\sim} &\mbox{ for } k>0.
\end{cases}
$$
In other words, $P/{\sim}^k$ is the poset obtained by applying the tree relation  $k$ times to $P$.  See Figure~\ref{fig:successiveQuotients} for an example of a poset and its successive quotients.  We are now ready to give one of the main theorems of this section. Recall that $\ell(P)$  denotes the length of $P$.

\begin{figure}
 \begin{tikzpicture}

\node[circle,  fill=black,  inner sep=0, minimum size= 7 pt] (a) at (-2,0) {};
\node[circle, fill=black,  inner sep=0, minimum size= 7 pt] (b) at (-1,0) { };
 \node[circle, fill=black,  inner sep=0, minimum size= 7 pt] (c) at (0,0) { };
  \node[circle, fill=black,  inner sep=0, minimum size= 7 pt] (d) at (1,0) { };
 \node[circle, fill=black,  inner sep=0, minimum size= 7 pt] (e) at (2,0) { };

\node[circle,  fill=black,  inner sep=0, minimum size= 7 pt] (f) at (-1,1) {};
  \node[circle, fill=black,  inner sep=0, minimum size= 7 pt] (g) at (.5,1) { };

\draw [thick] (a)--(f);
\draw [thick] (b)--(f);
\draw [thick] (c)--(f);
\draw [thick] (c)--(g);
\draw [thick] (d)--(g);
 
 \node[circle, fill=black,  inner sep=0, minimum size= 7 pt] (i) at (-1,2) { };
  \node[circle, fill=black,  inner sep=0, minimum size= 7 pt] (j) at (.5,2) { };
 \node[circle, fill=black,  inner sep=0, minimum size= 7 pt] (k) at (2,2) { };

\draw [thick] (f)--(i);
\draw [thick] (g)--(i);
\draw [thick] (g)--(j);
\draw [thick] (g)--(k);
 \draw [thick] (e)--(k);

 \node[circle, fill=black,  inner sep=0, minimum size= 7 pt] (l) at (.5,3) { };
 \node[circle, fill=black,  inner sep=0, minimum size= 7 pt] (m) at (2,3) { };

\draw [thick] (i)--(l);
\draw [thick] (j)--(l);
\draw [thick] (j)--(m);
\draw [thick] (k)--(l);
\draw [thick] (k)--(m);

 \node[circle, fill=black,  inner sep=0, minimum size= 7 pt] (1) at (1.25,4) { };

\draw [thick] (l)--(1);
\draw [thick] (m)--(1);

\node at (0,-1) {$P$};

\begin{scope} [shift={(4,-1)}]

\node[circle,  fill=black,  inner sep=0, minimum size= 7 pt] (fg) at (.5,1) {};
  \node[circle, fill=black,  inner sep=0, minimum size= 7 pt] (h) at (2,1) { };

 \node[circle, fill=black,  inner sep=0, minimum size= 7 pt] (i) at (-1,2) { };
  \node[circle, fill=black,  inner sep=0, minimum size= 7 pt] (j) at (.5,2) { };
 \node[circle, fill=black,  inner sep=0, minimum size= 7 pt] (k) at (2,2) { };

\draw [thick] (fg)--(i);
\draw [thick] (fg)--(i);
\draw [thick] (fg)--(j);
\draw [thick] (fg)--(k);
\draw [thick] (h)--(k);

 \node[circle, fill=black,  inner sep=0, minimum size= 7 pt] (l) at (.5,3) { };
 \node[circle, fill=black,  inner sep=0, minimum size= 7 pt] (m) at (2,3) { };

\draw [thick] (i)--(l);
\draw [thick] (j)--(l);
\draw [thick] (j)--(m);
\draw [thick] (k)--(l);
\draw [thick] (k)--(m);

 \node[circle, fill=black,  inner sep=0, minimum size= 7 pt] (1) at (1.25,4) { };

\draw [thick] (l)--(1);
\draw [thick] (m)--(1);

\node at (0.5,0) {$P/{\sim}$};

\end{scope}
\begin{scope} [shift={(6.5,-2)}]

   \node[circle, fill=black,  inner sep=0, minimum size= 7 pt] (ijk) at (1.25,2) { };

 \node[circle, fill=black,  inner sep=0, minimum size= 7 pt] (l) at (.5,3) { };
 \node[circle, fill=black,  inner sep=0, minimum size= 7 pt] (m) at (2,3) { };

\draw [thick] (ijk)--(l);
 \draw [thick] (ijk)--(m);

 \node[circle, fill=black,  inner sep=0, minimum size= 7 pt] (1) at (1.25,4) { };

\draw [thick] (l)--(1);
\draw [thick] (m)--(1);

\node at (1.25,1) {$P/{\sim}^2$};

\end{scope}

\begin{scope} [shift={(8.5,-3)}]

 \node[circle, fill=black,  inner sep=0, minimum size= 7 pt] (lm) at (1.25,3) { };

 \node[circle, fill=black,  inner sep=0, minimum size= 7 pt] (1) at (1.25,4) { };

\draw [thick] (lm)--(1);

\node at (1.25,2) {$P/{\sim}^3$};

\end{scope}

\end{tikzpicture}

\caption{Poset $P$ and its successive quotients.}\label{fig:successiveQuotients}

\end{figure}
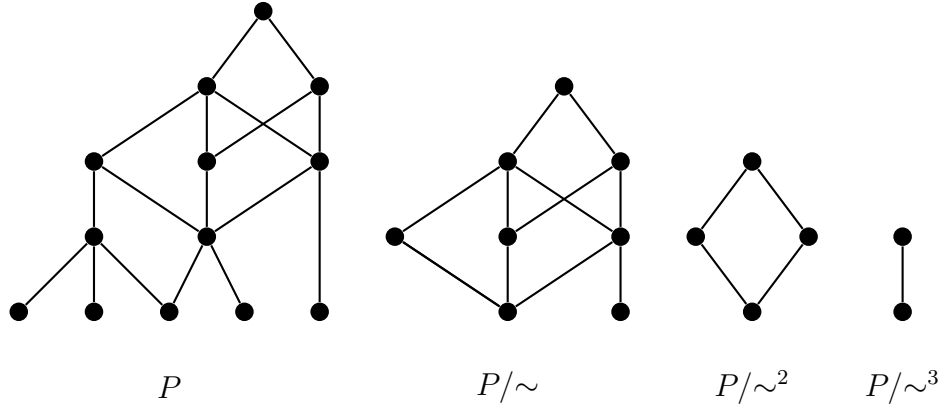

 \begin{theorem}\label{thm:characterizationOfRankedPosetGor}
 Let $P$ be a poset with a $\hat{1}$.  $U_P$ is Gorenstein if and only for all $0\leq k\leq \ell(P)-1$ the following hold
 \begin{itemize}
     \item [(1)] $M_{P/{\sim}^k}$ is acyclic.
     \item[(2)] $P/\sim^k$ satisfies the tree downset condition.
 \end{itemize}
\end{theorem}
\begin{proof}
 ($\Rightarrow$)  Suppose that $U_P$ is Gorenstein.  We prove assumptions (1) and (2) hold for all $0\leq k\leq \ell(P)-1$ by using strong induction on $k$.  When $k=0$, the result holds by Lemma~\ref{lem:NecessaryCondForGor} and Lemma~\ref{lem:ccConditionEqTotdCondition}.  

Now suppose that $k>0$.  By the inductive hypothesis, assumptions (1) and (2) hold for $P/{\sim}^{j}$ for all $0\leq j \leq k-1$.  By Lemma~\ref{lem:ccConditionEqTotdCondition}, (2) implies that $P/{\sim}^{j}$  satisfies the cc condition for all $0\leq j \leq k-1$.    Thus, we can repeatedly apply Lemma~\ref{lem:PGorIffPsimGor} to get that $U_P$ Gorenstein implies that $U_{P/{\sim}^{k-1}}$ is Gorenstein.  Applying the lemma to $P/{\sim^{k-1}}$ implies that $U_{P/{\sim}^k}$ is Gorenstein. Then  assumptions (1) and (2) hold for $P/{\sim}^k$ by Lemma~\ref{lem:NecessaryCondForGor}.  

($\Leftarrow$) Now suppose assumptions (1) and (2) hold for all $0\leq k\leq \ell(P)-1$.    Then we can repeatedly apply Lemma~\ref{lem:PGorIffPsimGor} to see that $U_P$ is Gorenstein if and only $U_{P/{\sim}^{\ell(P)-1}}$ is Gorenstein.  Now every time we apply the tree relation to poset its0 length drops by 1.  Thus, applying  the quotient $\ell(P)-1$ times, we are left with a poset consisting of a $\hat{1}$ and minimal elements.  Clearly such a poset gives rise of a Gorenstein variety. Thus, $U_{P/{\sim}^{\ell(P)-1}}$ is Gorenstein and so $U_P$ is Gorenstein as well.
\end{proof}

The previous theorem allows us to give an algorithm to determine if a poset with a $\hat{0}$ or $\hat{1}$ gives rise to a Gorenstein variety.

\begin{algorithm}\label{alg:RankedPosAlg}
 {\it  Algorithm to determine if $U_P$ is Gorenstein}\\
{\bf Input}: $P$ a  poset with a $\hat{0}$ or $\hat{1}$.\\
{\bf Output}: A Yes/No decision if $U_P$ is Gorenstein.\\
{\bf Method}:  
\begin{enumerate}
    \item If $P$ has a $\hat{0}$, set $P$ to be $P^\star$.
    \item While $\ell(P)>1$
    \begin{enumerate}
        \item Check the truth value of the following statement. 
        \begin{center}
            \textit{ $M_P$ is acyclic.}
        \end{center}
        If the statement is false, terminate algorithm and return ``No".
        \item Check the truth value of the following statement. 
       \begin{center}
            \textit{ $P$ satisfies the tree downset condition.}
        \end{center}
        If the statement is false, terminate algorithm and return ``No".
        \item Set $P$ to be  $P/{\sim}$, where $\sim$ is the tree relation.
    \end{enumerate}
\item Return ``Yes".
\end{enumerate}

\end{algorithm}

At each step of the algorithm, we need to check the tree downset condition.  Since this condition depends on the number of connected components in $M_P$, it is important to understand how the number of connected components change as we take  quotients.

\begin{lemma}\label{lem:numCCsAfterQuo}
Let $P$ be a poset with $k$ trees  in $M_P$. Then $M_{P/{\sim}}$ has at most $k$ trees.
\end{lemma}
\begin{proof}
Let $T_1,T_2,\dots, T_k$ be the trees of $M_P$. Then $T_1,T_2,\dots, T_k$  are the minimal elements of $P/{\sim}$.  Every connected component of $M_{P/{\sim}}$  must contain a minimum element of $P/{\sim}$.   So every connected component of $M_{P/{\sim}}$  contains some tree in $T_1,T_2,\dots, T_k$.   Moreover,  each $T_i$ is in exactly one connected component of  $M_{P/{\sim}}$. It follows that there are at most $k$ trees in $M_{P/{\sim}}$.
\end{proof}

In the next theorem, we give a bound on the complexity of Algorithm~\ref{alg:RankedPosAlg}.
Note that, the brute force method which would require first enumerating all downsets is known to be \textsf{\#P}-complete as we have previously remarked.
So, if the number of connected components of $M_P$ is small compared to $|P|$, the algorithm can be significantly quicker than the brute force approach (assuming \textsf{\#P}-complete problems are not actually easy).
Toward formalizing this, we analyze Algorithm~\ref{alg:RankedPosAlg} in terms of parameterized complexity~\cite{DF99}.
The parameterized complexity class \XP{} is consists of languages that are pairs of a problem instance and a parameter $k$ such that there is an algorithm solving the problem with time bounded by $f(k)\cdot N^{g(k)}$ where $N$ is the size of the problem instance and both $f$ and $g$ are computable functions.
We call an algorithm demonstrating membership in \XP{} an \emph{\XP{}-algorithm}.
In our situation a problem instance will be given by a poset $P$ on $n$ elements.
The parameter $k$ we use is the number of connected components of $M_P$.
We note that this parameter can be efficiently computed from the poset.

\begin{theorem}
Algorithm~\ref{alg:RankedPosAlg} is a valid \XP{}-algorithm which can be made to run in $O(kn^{k+3})$ time where $n=|P|$ and $k$ is the number of connected components of $M_P$.
\end{theorem}

\begin{proof}
The validity of the algorithm follows from Theorem~\ref{thm:characterizationOfRankedPosetGor} and the fact that if $P=M_P$, then $U_P$ is Gorenstein.   

Now we show the result concerning the time complexity.  First, the algorithm needs to determine $M_P$.  This can be done by first finding all the minimal elements of $P$ which can be done in $O(n^2)$ time by comparing elements.  Then one must find the elements that cover only minimal elements. This too can be done in $O(n^2)$ time.   Because we will need it later, we can also find the connected components of $M_P$ again in $O(n^2)$ time using a depth first search.
Next, the algorithm needs to determine if $M_P$ is acyclic.   Since we know the number of connected components of $M_P$, we can just count the number of edges in $M_P$ to decide this.

After determining if $M_P$ is acyclic, the algorithm must check the tree downset condition of $P$ (Definition~\ref{def:treeDownsetCond}).  Let $S\subseteq P\setminus M_P$ with $|S|\leq k$ where $k$ is the number of connected components of $M_P$.  To check the tree downset condition for $S$, we must find all the connected components that intersect nontrivially with $\downset{S}$.
We can compute $\downset{S}$ in $O(kn)$ time by comparing each element  of $P$ with each element of $S$ and from these comparisons we can also determine which connected components of $M_P$ intersect nontrivially with $\downset{S}$.

Next, for each $j$ we will compute  $A_j$ given in Definition~\ref{def:treeDownsetCond} and check if it is connected. 
Computing $A_j$ can be done in linear time given that we have already computed $\downset{S}$.
Again, using a depth first search, determining connectedness of $A_j$ can be done in $O(n^2)$ time. For those that are connected, we  need to count the number of connected components of $A_j\cap M_P$.  This can be done in $O(n^2)$ time. Since there at most $k$ possible $A_j$'s to check, we see that for each set $S$, checking the tree downset condition for $S$, can be done in $O(kn^2)$ time.   Now there are $O(n^k)$ sets to check, so this step can be done in $O(kn^{k+2})$ time.
Lastly, if the tree downset condition is satisfied we must compute $P/{\sim}$. This can be handled by recording some of the computations performed above since $P/{\sim}$ is determined by comparing elements of $M_P$ with elements of $P \setminus M_P$.

We conclude that one iteration of steps (1) and (2) of the algorithm can be run in $O(kn^{k+2})$ time.  By Lemma~\ref{lem:numCCsAfterQuo}, the number of connected components of $M_{P/{\sim^{j}}}$ is at most $k$ and so each iteration can be run in $O(kn^{k+2})$ time.   Since the algorithm runs no more than $n$ times, we see that the full algorithm can run in $O(kn^{k+3})$ time as claimed.
\end{proof}

We finish this section by noting that when $M_P$ is connected, we can easily check if $U_P$ is Gorenstein.  The main reason for the simplicity  is that if we quotient out by the tree relation and $M_P$ is connected, $P/{\sim}$ is bounded and so $U_{P/{\sim}}$ is Gorenstein.  We leave the details of the proof for the following theorem to the reader.

\begin{theorem}
 Let $P$ be a poset with a $\hat{1}$ such that $M_P$ is connected.  Then the following are equivalent.
 \begin{enumerate}
     \item[(a)] $U_P$ is Gorenstein.
     \item[(b)] For all $x\in P$,  $\downset{x} \cap M_P$ is a tree.
       \item[(c)] For all $x\in  P\setminus M_P$, $\mu_{\widehat{P}}(x) = 0$.
 \end{enumerate}
\end{theorem}

\section{Length 1 posets}
 \label{sec:len1}
 
 In this section we restrict our attention to posets of length $1$.
 That is, posets in which every element is either minimal or maximal.

\subsection{$\mathbb{Q}$-Gorenstein implies Gorenstein}

We start with a lemma which determines the values an $r$-Gorenstein labeling must take on any biconnected length $1$ poset.

\begin{lemma}\label{lem:bimaxmin}
If $P$ is a length 1 biconnected  poset and $\phi:P \to \mathbb{Z}$ is an $r$-Gorenstein labeling, then $\phi(x) = r$ for all maximal elements $x \in P$ and $\phi(y) = -r$ for all minimal elements $y \in P$.
\end{lemma}

\begin{proof}
For each maximal element $x \in P$ we have that $A = \{x\}$ is an upset.
Since $P$ is biconnected, $\dim(A) = 1$ so $\phi(x) = r$ because $\phi$ is an $r$-Gorenstein labeling.  Similarly, by Proposition~\ref{prop:downsetLabel}, for any minimal $y\in P$, we must have that $\phi(y)=-r$
\end{proof}

\begin{proposition}
If $P$ is a length $1$ biconnected poset with an $r$-Gorenstein labeling, then $P$ has the same number of minimal and maximal elements.
\label{prop:len1}
\end{proposition}
\begin{proof}
For a length $1$ poset each element is either a minimal or maximal element and cannot be both since the poset is connected.
Thus if $\phi$ is an $r$-Gorenstein labeling its values are completely determined by Lemma~\ref{lem:bimaxmin}.
Let $\alpha$ be the number of maximal elements of $P$ and $\beta$ denote the number of minimal elements.
Then we have that
\[0 = \sum_{z \in P} \phi(z) = \alpha r - \beta r.\]
So $\alpha = \beta$ and the result is proven.
\end{proof}

Let us give an example which shows that the converse of Proposition~\ref{prop:len1} is false.
The length $1$ poset in Figure~\ref{fig:len1} is biconnected with the same number of minimal and maximal elements, but is not $r$-Gorenstein.  To see why, note that all the maximal elements would have to have to be labeled with a value of $r$.
Now let $A$ be the upset generated by $2$, then $cc(A)=1=cc(\overline{A})$ and so it follows that $2$ must be labeled by $-2r$.
Thus by Lemma~\ref{lem:bimaxmin}, the poset cannot be $r$-Gorenstein.

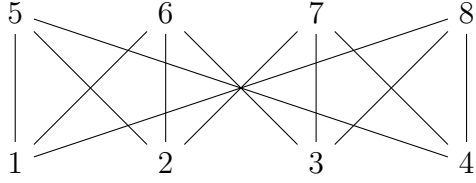
\begin{figure}
    \centering
\begin{center}
\begin{tikzpicture}
\node(1) at (0,0) {$1$};
\node(2) at (2,0) {$2$};
\node(3) at (4,0) {$3$};
\node(4) at (6,0) {$4$};

\node(5) at (0,2) {$5$};
\node(6) at (2,2) {$6$};
\node(7) at (4,2) {$7$};
\node(8) at (6,2) {$8$};

\draw (1)--(5);
\draw (1)--(6);
\draw (1)--(8);

\draw (2)--(5);
\draw (2)--(6);
\draw (2)--(7);

\draw (3)--(6);
\draw (3)--(7);
\draw (3)--(8);

\draw (4)--(7);
\draw (4)--(8);
 \draw(4)--(5);
\end{tikzpicture}
\end{center}
    \caption{A biconnected poset of length $1$ which has the same number of minimal and maximal elements but is not Gorenstein.}
    \label{fig:len1}
\end{figure}

We now give one of the main theorems of this section.

\begin{theorem}\label{thm:GorIffrGor}
 Let $P$ be a length $1$ poset. Then $U_P$ is Gorenstein if and only if $U_P$ is $\mathbb{Q}$-Gorenstein.\end{theorem}
\begin{proof}
The backwards direction is trivial given the definitions of $\mathbb{Q}$-Gorenstein and Gorenstein. By Theorem~\ref{thm:bicon} we may assume that $P$ is biconnected. So by Lemma~\ref{lem:bimaxmin}, the $r$-Gorenstein labeling is given by $\phi(x) =\pm r$.  It follows that the labeling $\dfrac{1}{r}\phi(x)$ is a Gorenstein labeling of $P$.
\end{proof}
 
\subsection{Some infinite families}

 By Theorem~\ref{thm:GorIffrGor} we need to only look  at Gorenstein labelings. We will now give some infinite families  of length $1$ posets admitting Gorenstein labelings.

\begin{proposition}\label{prop:2reg}
Let $P$ be a length 1 poset which is  $2$-regular and $2$-connected. Then $U_P$ is Gorenstein. 
\end{proposition}
\begin{proof}
For such a poset $P$ it must be that $P$  is a $2k$-cycle.
This is a balanced bipartite graph with $k$ elements in each part of the bipartition.
The bipartition splits $P$ into its minimal elements and its maximal elements.
Let $\phi: P \rightarrow \mathbb{Z}$ be defined by 
$$
\phi(x) = 
\begin{cases}
1 & \mbox{ if  $x$ is maximal} \\
-1 & \mbox{ if  $x$ is minimal} 
\end{cases}
$$
which is the only possibility for a Gorenstein labeling by Lemma~\ref{lem:bimaxmin}.
First, note that since there are $k$ minimal elements and $k$ maximal elements, we have that 
$$
\sum_{x\in P} \phi(x) = k-k=0.
$$
Now suppose that $A$ is an upset such that $A\neq P$ and $\dim(A) =1$.  Then $A$ is connected.  The only proper connected subgraphs of a cycle are paths and so $A$ is a path.  Since $A$ is an upset and a path, it must be that there are $j$ minimal elements in the path  and $j+1$ maximal elements.  It follows that for any upset $A$ which is connected, 
$$
\sum_{x\in A} \phi(x) = (j+1)-j=1.
$$
Thus $U_P$ is Gorenstein.
\end{proof}

As one can see from the poset in Figure~\ref{fig:len1}, we cannot just replace 2-regular with 3-regular in the previous proposition. However, there are $n$-regular length 1 posets which give  rise to Gorenstein varieties as we see next.  Throughout the remainder of this section, we use $K_{n,n}$ to denote the poset whose Hasse diagram is the complete bipartite graph with $n$ minimal elements and $n$ maximal elements.

\begin{proposition}\label{prop:nreg}
If $P$ is $K_{n,n}$, then $U_P$ is Gorenstein. 
\end{proposition}
\begin{proof}
First note that $P$ has $n$ minimal elements each of which is less than each of the $n$ maximal elements.
 Let the minimal elements of $P$ be $a_1$, $a_2$, $\dots$, $a_n$.
Also, let the maximal elements of $P$ be $b_1$, $b_2$, $\dots$, $b_n$.
Let $\phi: P \rightarrow \mathbb{Z}$ be given by $\phi(a_i) = -1$ and $\phi(b_i) = 1$ for all $1 \leq i \leq n$
which is the only possibility for a Gorenstein labeling by Lemma~\ref{lem:bimaxmin}.
It is clear that $\sum_{x \in P} \phi(x) = 0$.
Now if $A \subseteq P$ is an upset with $\dim(A) = 1$, then either $A = \{b_i\}$ for some $1 \leq i \leq n$ or else $\overline{A} = \{a_i\}$ again for some $1 \leq i \leq n$.
In any case $\sum_{x \in A} \phi(x) = 1$ for any upset $A \subseteq P$ with $\dim(A) = 1$.
Therefore $U_P$ is Gorenstein.
\end{proof}

By Proposition~\ref{prop:nreg}, we know that there are $n$-regular, $n$-connected posets which have Gorenstein labelings, namely $K_{n,n}$.  Additionally, by Proposition~\ref{prop:2reg}, we know that there are $2$-regular, 2-connected posets which give rise to Gorenstein varieties other than $K_{2,2}$.    This naturally brings up the question of if there are $n$-regular, $n$-connected posets, other than $K_{n,n}$, whose  variety is Gorenstein when   $n\geq 3$.  We have been able verify that $K_{3,3}$ is the only $3$-regular, 3-connected length 1 poset whose variety  is Gorenstein.  However, at this time our methods do not need seem to generalize  when $n\geq 4$ and thus this case remains open.

\section{Conclusion}
\label{sec:conclusion}

We now conclude with some discussion of open problems.
First we have the following question on recognizing when an $r$-Gorenstein labeling exists.

\begin{question}
\label{q:G}
Is there a good characterization in terms of $P$ of when $U_P$ is Gorenstein or $\mathbb{Q}$-Gorenstein?
\end{question}

In Question~\ref{q:G} the word ``good'' means a characterization which is not simply the definition of being ($\mathbb{Q}$)-Gorenstein and preferably gives an efficient means of recognizing the Gorenstein property from the poset like is done in Algorithm~\ref{alg:RankedPosAlg} for posets with $\hat{0}$ or $\hat{1}$ when the number of connected components of $M_P$ is small.
A braid cone is smooth if and only if the Hasse diagram of the poset is a tree~\cite[Corollary 3.10]{PRW}.
Since smooth implies Gorenstein, tree posets will be included in any characterization that answers Question~\ref{q:G}.
Another easy to describe class of posets we know have a Gorenstein labeling is bounded posets from Theorem~\ref{thm:crepant}.
The largest class   we understand, which contains tree posets and bounded posets, comes from the characterization in Theorem~\ref{thm:characterizationOfRankedPosetGor} combined with the reduction to biconnected components from Theorem~\ref{thm:bicon}.
A question, possibly more approachable than in the general case, is the following for which we treated some cases of in Section~\ref{sec:len1}.

\begin{question}
\label{q:len1}
Is there an efficient way to decide if $U_P$ is Gorenstein  when $P$ is a length $1$ poset?
\end{question}

By Proposition~\ref{prop:len1} and reduction to the biconnected case previously mentioned, we see that Question~\ref{q:len1} becomes a problem about neighborhoods in balanced bipartite graphs with a ``top'' and ``bottom'' set of vertices corresponding to maximal and minimal elements respectively. 
Approaches to this question could be to look for a clever algorithm or look for a reduction with an NP-complete problem.

Whether it is the length 1 case or the general case, in order to decide if $U_P$ is Gorenstein, we need understand the ray generators of $\sigma_P$. As we mentioned before, doing this may be difficult since the ray generators come from upsets which are known to be hard to enumerate~\cite{ProvanBall1983}. This makes the problem of deciding Gorensteinness potentially difficult for generic posets. 

In Section~\ref{sec:posetsWith0or1} and Section~\ref{sec:len1} we showed that $\mathbb{Q}$-Gorenstein implies Gorenstein for posets with a $\hat{0}$ or a $\hat{1}$ and for length 1 posets. Moreover, using SageMath~\cite{SageMath}, we have been able to verify that for all posets on up to $10$ elements, $\mathbb{Q}$-Gorenstein implies Gorenstein. This leads us to the following  conjecture. 

\begin{conjecture}
\label{conj:QimpliesG}
Let $P$ be a poset, then $U_P$ is Gorenstein if and only if $U_P$ is $\mathbb{Q}$-Gorenstein.
\end{conjecture}
 
We note that braid cones  have a nice property related to smoothness of a similar flavor to this conjecture.
In particular, a braid cone is simplicial (i.e.~ray generators are a basis of $N_{\mathbb{R}}$) if and only if it is smooth (i.e~ray generators are a basis of $N$)~\cite[Corollary 3.10]{PRW}.
Geometrically this means that the corresponding toric variety is an orbifold if and only if it is smooth.

\bibliographystyle{alphaurl}
\bibliography{refs}

\end{document}